\documentclass[12pt]{article}

\usepackage[utf8]{inputenc}
\usepackage{amssymb,latexsym}
\usepackage{enumerate}

\usepackage[normalem]{ulem} 

\usepackage[dvips]{graphicx}
\usepackage{pstricks} 
\usepackage{pst-optic}
\usepackage{pst-text}
\usepackage{amsmath}
\usepackage{color}

\usepackage{tikz}
\usetikzlibrary{arrows}

\usepackage{hyperref}

\newcommand{\kw}{\rule{2mm}{2mm}}

\newcommand{\qed}{\hfill \kw}
\newcommand{\pc}{\overrightarrow{pc}}
\newcommand{\wc}{\overrightarrow{wc}}
\newcommand{\Ri}{\curvearrowright}
\newcommand{\ri}{\rightarrow}
\newcommand{\Rri}{\leadsto}
\newcommand{\dist}{\mathrm{dist}}
\newcommand{\vf}{\varphi}

\newenvironment{proof}{\noindent\textbf{Proof.}}{\hfill$\qed$\medskip} 
\newtheorem{theorem}{Theorem}
\newtheorem{lemma}[theorem]{Lemma}

\newtheorem{prop}[theorem]{Proposition}
\newtheorem{obs}[theorem]{Observation}
\newtheorem{conjecture}{Conjecture}
\newtheorem{corollary}[theorem]{Corollary}

\newtheorem{question}{Question} 

\newcommand{\caseTwo}
{
\begin{tikzpicture}[>=stealth']
\draw [thick,dotted](0,0) circle (3);

\draw [fill] (80:3) circle (0.05);
\draw [fill] (10:3) circle (0.05);
\draw [fill] (-60:3) circle (0.05);
\draw [fill] (-120:3) circle (0.05);
\draw [fill] (-140:3) circle (0.05);
\draw [fill] (100:3) circle (0.05);

\draw [fill] (-10:3) circle (0.05);
\draw [fill] (-40:3) circle (0.05);
\draw [fill] (140:3) circle (0.05);
\draw[->,thick,red,double,dashed](140:3)--(-10:3);
\draw[->,thick,red,double,dashed] (-10:3) arc (-10:-40:3);
\draw[->,thick,red,double,dashed] (-40:3) arc (-40:-60:3);

\draw[->,thick] (80:3) arc (80:10:3);

\draw[->,thick] (-60:3) arc (-60:-120:3);

\draw[->,thick] (220:3) arc (220:100:3);

\draw[->,thick](100:3)--(-60:3);

\draw[->,thick](-120:3)--(80:3);

\draw[->,thick](10:3)--(220:3);

\node [above] at (80:3) {$x_1$};
\node  [above] at (100:3) {$x_0$};

\node  [right] at (10:3) {$x_{r+2}$};
\node  [below] at (-60:3) {$x_k$};
\node  [below] at (-120:3) {$x_{k+r+1}$};
\node  [left] at (220:3) {$x_{k+r+2}$};

\node  [right] at (-10:3) {$x_{r+3}$};
\node  [right] at (-40:3) {$x_{k-1}$};
\node  [left] at (140:3) {$x_{k+2r+3}$};
\end{tikzpicture}
}

\newcommand{\caseThreeTwo}
{
\begin{tikzpicture}[>=stealth']
\draw [thick,dotted](0,0) circle (3);

\draw [fill] (0:3) circle (0.05);
\draw [fill] (20:3) circle (0.05);
\draw [fill] (50:3) circle (0.05);
\draw [fill] (70:3) circle (0.05);
\draw [fill] (90:3) circle (0.05);
\draw [fill] (140:3) circle (0.05);
\draw [fill] (160:3) circle (0.05);
\draw [fill] (-100:3) circle (0.05);
\draw [fill] (-80:3) circle (0.05);
\draw [fill] (-20:3) circle (0.05);

\draw[->,thick] (280:3) arc (280:160:3);
\draw[->,thick](160:3)--(0:3);
\draw[->,thick](0:3)--(90:3);
\draw[->,thick](90:3)--(-80:3);

\draw[->,thick,red,double,dashed] (50:3) arc (50:20:3);
\draw[->,thick,red,double,dashed] (20:3) arc (20:0:3);
\draw[->,thick,red,double,dashed](280:3)--(50:3);

\draw[->,thick,green,double] (160:3) arc (160:140:3);
\draw[->,thick,green,double](140:3)--(-20:3);
\draw[->,thick,green,double](-20:3)--(70:3);
\draw[->,thick,green,double](70:3)--(-100:3);

\node  [above] at (90:3) {$x_{0}$};
\node  [right] at (0:3) {$x_{n-k}$};
\node  [right] at (-80:3) {$x_{n-k+s+2}$};
\node  [left] at (160:3) {$x_{2(n-k)}$};
\node  [right] at (50:3) {$x_{s+2}$};
\node  [right] at (20:3) {$x_{n-k-1}$};
\node  [left] at (-100:3) {$x_{n-k+s+3}$};
\node  [left] at (140:3) {$x_{2(n-k)+1}$};
\node  [above] at (70:3) {$x_{1}$};
\node  [right] at (-20:3) {$x_{n-k+1}$};

\end{tikzpicture}
}

\newcommand{\caseThreeFive}
{
\begin{tikzpicture}[>=stealth']
\draw [thick,dotted](0,0) circle (3);

\draw [fill] (40:3) circle (0.05);
\draw [fill] (90:3) circle (0.05);
\draw [fill] (110:3) circle (0.05);
\draw [fill] (160:3) circle (0.05);
\draw [fill] (180:3) circle (0.05);
\draw [fill] (220:3) circle (0.05);
\draw [fill] (260:3) circle (0.05);
\draw [fill] (300:3) circle (0.05);
\draw [fill] (350:3) circle (0.05);
\draw [fill] (280:3) circle (0.05);
\draw [fill] (200:3) circle (0.05);

\draw[->,dashed] (300:3) arc (300:280:3);
\draw[->,dashed] (280:3) arc (280:260:3);

\draw[->,dashed] (220:3) arc (220:200:3);
\draw[->,dashed] (200:3) arc (200:180:3);

\draw[->,thick] (160:3) arc (160:110:3);
\draw[->,thick](90:3)--(160:3);
\draw[->,thick](110:3)--(180:3);

\draw[->,thick,red,double,dashed] (260:3) arc (260:220:3);
\draw[->,thick,red,double,dashed] (180:3) --(260:3);
\draw[->,thick,red,double,dashed] (220:3)--(300:3);

\draw[->,thick,green,double](300:3)--(350:3);
\draw[->,thick,green,double](350:3)--(40:3);
\draw[->,thick,green,double](40:3)--(90:3);

\node  [above] at (90:3) {$x_{0}$};
\node  [right] at (40:3) {$x_{n-k}$};
\node  [right] at (350:3) {$x_{2(n-k)}$};
\node  [right] at (300:3) {$x_{3(n-k)}$};
\node  [below] at (260:3) {$x_{3(n-k)+2}$};
\node  [left] at (220:3) {$x_{4(n-k)}$};
\node  [left] at (180:3) {$x_{4(n-k)+2}$};
\node  [left] at (160:3) {$x_{4(n-k)+3}$};
\node  [left] at (110:3) {$x_{5(n-k)+2}$};

\end{tikzpicture}
}

\newcommand{\caseThreeFour}
{
\begin{tikzpicture}[>=stealth']
\draw [thick,dotted](0,0) circle (3);

\draw [fill] (90:3) circle (0.05);
\draw [fill] (110:3) circle (0.05);
\draw [fill] (160:3) circle (0.05);
\draw [fill] (180:3) circle (0.05);
\draw [fill] (200:3) circle (0.05);
\draw [fill] (250:3) circle (0.05);
\draw [fill] (270:3) circle (0.05);
\draw [fill] (310:3) circle (0.05);
\draw [fill] (340:3) circle (0.05);
\draw [fill] (20:3) circle (0.05);

\draw[->,thick] (160:3) arc (160:110:3);
\draw[->,thick](90:3)--(160:3);
\draw[->,thick](110:3)--(180:3);
\draw[->,thick] (250:3) arc (250:200:3);
\draw[->,thick](180:3)--(250:3);
\draw[->,thick](200:3)--(270:3);

\draw[->,thick,red,double,dashed] (340:3) arc (340:310:3);
\draw[->,thick,red,double,dashed] (270:3) --(340:3);
\draw[->,thick,red,double,dashed](310:3)--(20:3);

\draw[->,thick,green,double](20:3)--(90:3);

\node  [above] at (90:3) {$x_{0}$};
\node  [above] at (110:3) {$x_{4(n-k)+s+3}$};
\node  [right] at (20:3) {$x_{n-k}$};
\node  [right] at (340:3) {$x_{n-k+s+2}$};
\node  [right] at (310:3) {$x_{2(n-k)}$};
\node  [below] at (270:3) {$x_{2(n-k)+s+2}$};
\node  [left] at (250:3) {$x_{2(n-k)+s+3}$};
\node  [left] at (200:3) {$x_{3(n-k)+s+2}$};
\node  [left] at (180:3) {$x_{3(n-k)+s+3}$};
\node  [left] at (160:3) {$x_{3(n-k)+s+4}$};

\end{tikzpicture}
}

\newcommand{\caseThreeOne}
{
\begin{tikzpicture}[>=stealth']
\draw [thick,dotted](0,0) circle (3);

\draw [fill] (90:3) circle (0.05);
\draw [fill] (110:3) circle (0.05);
\draw [fill] (160:3) circle (0.05);
\draw [fill] (180:3) circle (0.05);
\draw [fill] (200:3) circle (0.05);
\draw [fill] (250:3) circle (0.05);
\draw [fill] (270:3) circle (0.05);
\draw [fill] (290:3) circle (0.05);
\draw [fill] (340:3) circle (0.05);
\draw [fill] (0:3) circle (0.05);

\draw[->,thick] (160:3) arc (160:110:3);
\draw[->,thick](90:3)--(160:3);
\draw[->,thick](110:3)--(180:3);
\draw[->,thick] (250:3) arc (250:200:3);
\draw[->,thick](180:3)--(250:3);
\draw[->,thick](200:3)--(270:3);
\draw[->,thick] (340:3) arc (340:290:3);
\draw[->,thick](270:3)--(340:3);
\draw[->,thick](290:3)--(0:3);

\draw[->,thick,green,double](0:3)--(90:3);

\node  [above] at (90:3) {$x_{0}$};
\node  [above] at (110:3) {$x_{4(n-k)+2}$};
\node  [right] at (0:3) {$x_{n-k}$};
\node  [right] at (340:3) {$x_{n-k+1}$};
\node  [right] at (290:3) {$x_{2(n-k)}$};
\node  [below] at (270:3) {$x_{2(n-k)+1}$};
\node  [left] at (250:3) {$x_{2(n-k)+2}$};
\node  [left] at (200:3) {$x_{3(n-k)+1}$};
\node  [left] at (180:3) {$x_{3(n-k)+2}$};
\node  [left] at (160:3) {$x_{3(n-k)+3}$};

\end{tikzpicture}
}

\newcommand{\caseOneThreeOne}
{
\begin{tikzpicture}[>=stealth']
\draw [thick,dotted](0,0) circle (3);

\draw [fill] (90:3) circle (0.05);
\draw [fill] (70:3) circle (0.05);
\draw [fill] (30:3) circle (0.05);
\draw [fill] (10:3) circle (0.05);
\draw [fill] (330:3) circle (0.05);
\draw [fill] (310:3) circle (0.05);
\draw [fill] (250:3) circle (0.05);
\draw [fill] (230:3) circle (0.05);
\draw [fill] (210:3) circle (0.05);
\draw [fill] (170:3) circle (0.05);
\draw [fill] (150:3) circle (0.05);
\draw [fill] (130:3) circle (0.05);

\draw[->,thick,blue,dashed] (90:3)--(10:3);
\draw[->,thick,blue,dashed] (10:3) arc (10:-30:3);
\draw[->,thick,blue,dashed] (-30:3)--(250:3);

\draw[->,thick,] (70:3) arc (70:30:3);
\draw[->,thick] (30:3)--(310:3);
\draw[->,thick] (310:3)--(230:3);

\draw[->,thick,red,dashed,double] (230:3) arc (230:210:3);
\draw[->,thick,red,dashed,double] (210:3)--(130:3);
\draw[->,thick,red,dashed,double] (130:3) arc (130:90:3);

\draw[->,thick,green,double](250:3)--(170:3);
\draw[->,thick,green,double] (170:3) arc (170:150:3);
\draw[->,thick,green,double](150:3)--(70:3);

\node  [above] at (90:3) {$x_{0}$};
\node  [above] at (70:3) {$x_{1}$};
\node  [right] at (30:3) {$x_{k-1}$};
\node  [right] at (10:3) {$x_{k}$};
\node  [right] at (330:3) {$x_{2k-2}$};
\node  [right] at (310:3) {$x_{2k-1}$};
\node  [below] at (250:3) {$x_{3k-2}$};
\node  [below] at (230:3) {$x_{3k-1}$};
\node  [left] at (210:3) {$x_{3k}$};
\node  [left] at (170:3) {$x_{4k-2}$};
\node  [left] at (150:3) {$x_{4k-1}$};
\node  [left] at (130:3) {$x_{4k}$};
\end{tikzpicture}
}

\newcommand{\caseOneOne}
{
\begin{tikzpicture}[>=stealth']
\draw [thick,dotted](0,0) circle (3);

\draw [fill] (90:3) circle (0.05);
\draw [fill] (70:3) circle (0.05);
\draw [fill] (20:3) circle (0.05);
\draw [fill] (0:3) circle (0.05);
\draw [fill] (310:3) circle (0.05);
\draw [fill] (290:3) circle (0.05);
\draw [fill] (250:3) circle (0.05);
\draw [fill] (230:3) circle (0.05);
\draw [fill] (200:3) circle (0.05);
\draw [fill] (170:3) circle (0.05);
\draw [fill] (140:3) circle (0.05);
\draw [fill] (120:3) circle (0.05);

\draw[->,thick,blue,dashed] (90:3)--(20:3);
\draw[->,thick,blue,dashed] (20:3) -- (310:3);
\draw[->,thick,blue,dashed] (310:3)--(250:3);

\draw[->,thick,] (70:3) -- (0:3);
\draw[->,thick] (0:3)--(290:3);
\draw[->,thick] (290:3)--(230:3);

\draw[->,thick,green,double] (250:3)--(170:3);
\draw[->,thick,green,double] (170:3) arc (170:140:3);
\draw[->,thick,green,double] (140:3) -- (70:3);

\draw[->,thick,red,dashed,double] (230:3) arc (230:200:3);
\draw[->,thick,red,dashed,double](200:3)--(120:3);
\draw[->,thick,red,dashed,double](120:3) arc (120:90:3);

\node  [above] at (90:3) {$x_{0}$};
\node  [above] at (70:3) {$x_{1}$};
\node  [right] at (20:3) {$x_{k}$};
\node  [right] at (0:3) {$x_{k+1}$};
\node  [right] at (310:3) {$x_{2k}$};
\node  [below] at (290:3) {$x_{2k+1}$};
\node  [below] at (250:3) {$x_{3k}$};
\node  [left] at (230:3) {$x_{3k+1}$};
\node  [left] at (200:3) {$x_{3k+r+2}$};
\node  [left] at (170:3) {$x_{4k}$};
\node  [left] at (140:3) {$x_{4k+r+1}$};
\node  [left] at (120:3) {$x_{4k+r+2}$};

\end{tikzpicture}
}

\newcommand{\caseOneTwo}
{
\begin{tikzpicture}[>=stealth']
\draw [thick,dotted](0,0) circle (3);

\draw [fill] (90:3) circle (0.05);
\draw [fill] (70:3) circle (0.05);
\draw [fill] (355:3) circle (0.05);
\draw [fill] (335:3) circle (0.05);
\draw [fill] (260:3) circle (0.05);
\draw [fill] (240:3) circle (0.05);
\draw [fill] (170:3) circle (0.05);
\draw [fill] (150:3) circle (0.05);

\draw[->,thick,blue,dashed] (90:3)--(355:3);
\draw[->,thick,blue,dashed] (355:3) -- (260:3);
\draw[->,thick,blue,dashed] (260:3)--(170:3);

\draw[->,thick,] (70:3) -- (335:3);
\draw[->,thick] (335:3)--(240:3);
\draw[->,thick] (240:3)--(150:3);

\draw[->,thick,green,double] (170:3) -- (70:3);

\draw[->,thick,red,double,dashed] (150:3) arc (150:90:3);

\node  [above] at (90:3) {$x_{0}$};
\node  [above] at (70:3) {$x_{1}$};
\node  [right] at (355:3) {$x_{k}$};
\node  [right] at (335:3) {$x_{k+1}$};
\node  [below] at (260:3) {$x_{2k}$};
\node  [left] at (240:3) {$x_{2k+1}$};
\node  [left] at (170:3) {$x_{3k}$};
\node  [left] at (150:3) {$x_{3k+1}$};

\end{tikzpicture}
}

\newcommand{\caseOneThreeTwo}
{
\begin{tikzpicture}[>=stealth']
\draw [thick,dotted](0,0) circle (3);

\draw [fill] (90:3) circle (0.05);
\draw [fill] (80:3) circle (0.05);
\draw [fill] (70:3) circle (0.05);
\draw [fill] (350:3) circle (0.05);
\draw [fill] (340:3) circle (0.05);
\draw [fill] (330:3) circle (0.05);
\draw [fill] (320:3) circle (0.05);
\draw [fill] (310:3) circle (0.05);
\draw [fill] (230:3) circle (0.05);
\draw [fill] (220:3) circle (0.05);
\draw [fill] (210:3) circle (0.05);
\draw [fill] (200:3) circle (0.05);
\draw [fill] (190:3) circle (0.05);
\draw [fill] (110:3) circle (0.05);
\draw [fill] (100:3) circle (0.05);

\draw[->,thick,] (90:3) -- (320:3);
\draw[->,thick] (320:3)--(190:3);
\draw[->,thick] (190:3) arc (190:110:3);
\draw[->,thick] (110:3)--(340:3);
\draw[->,thick] (340:3)--(210:3);
\draw[->,thick] (210:3)--(90:3);

\draw[->,thick,green,double] (70:3) arc (70:-10:3);

\draw[->,thick,red,double,dashed] (310:3) arc (310:230:3);

\node  [above] at (90:3) {$x_{0}$};
\node  [above] at (80:3) {$x_{1}$};
\node  [right] at (70:3) {$x_{2}$};
\node  [right] at (350:3) {$x_{k-3}$};
\node  [right] at (340:3) {$x_{k-2}$};
\node  [right] at (320:3) {$x_{k}$};
\node  [right] at (310:3) {$x_{k+1}$};
\node  [left] at (230:3) {$x_{2k-4}$};
\node  [left] at (210:3) {$x_{2k-2}$};
\node  [left] at (190:3) {$x_{2k}$};
\node  [left] at (110:3) {$x_{3k-4}$};

\end{tikzpicture}
}

\newcommand{\CSevenThree}
{
\begin{tikzpicture}[>=stealth']

\draw [fill] (39:2) circle (0.05);
\draw [fill] (141:2) circle (0.05);
\draw [fill] (90:2) circle (0.05);
\draw [fill] (192:2) circle (0.05);
\draw [fill] (-12:2) circle (0.05);
\draw [fill] (243:2) circle (0.05);
\draw [fill] (-63:2) circle (0.05);

\draw[->,thick,blue] (90:2) -- (39:2);
\draw[->,thick,blue] (90:2)--(-63:2);
\draw[->,thick,blue] (141:2) -- (-12:2);
\draw[->,thick,blue] (192:2)--(142:2);
\draw[->,thick,blue] (192:2)--(39:2);
\draw[->,thick,blue] (243:2)--(90:2);
\draw[->,thick,blue] (-12:2)--(-63:2);
\draw[->,thick,blue] (39:2)--(243:2);

\draw[->,thick,red,dashed] (141:2) -- (90:2);
\draw[->,thick,red,dashed] (243:2) -- (192:2);
\draw[->,thick,red,dashed] (-63:2) -- (243:2);
\draw[->,thick,red,dashed] (-63:2) -- (141:2);
\draw[->,thick,red,dashed] (-12:2) -- (192:2);
\draw[->,thick,red,dashed] (39:2) -- (-12:2);

\end{tikzpicture}
}

\newcommand{\sectionThreeOne}
{
\begin{tikzpicture}[>=stealth']
\draw [thick,dotted](0,0) circle (3);

\draw[->,thick,red,double] (60:3) arc (60:45:3);
\draw[->,thick,red,double] (45:3) -- (0:3);
\draw[->,thick,red,double] (30:3) arc (30:15:3);
\draw[->,thick,red,double] (310:3) arc (310:295:3);
\draw[->,thick,red,double] (280:3) arc (280:265:3);
\draw[->,thick,red,double] (220:3) arc (220:205:3);
\draw[->,thick,red,double] (185:3) arc (185:170:3);
\draw[->,thick,red,double] (105:3) arc (105:90:3);
\draw[->,thick,red,double] (220:3)-- (325:3);

\draw[->,thick,] (90:3) arc (90:75:3);
\draw[->,thick,] (75:3) arc (75:60:3);
\draw[->,thick,] (45:3) arc (45:30:3);
\draw[->,thick,] (15:3) arc (15:0:3);
\draw[->,thick,] (325:3) arc (325:310:3);
\draw[->,thick,] (295:3) arc (295:280:3);
\draw[->,thick,] (235:3) arc (235:220:3);
\draw[->,thick,] (170:3) arc (170:155:3);
\draw[->,thick,] (170:3) -- (45:3);
\draw[->,thick] (90:3)--(280:3);

\draw [fill] (90:3) circle (0.05);
\draw [fill] (75:3) circle (0.05);
\draw [fill] (60:3) circle (0.05);
\draw [fill] (45:3) circle (0.05);
\draw [fill] (30:3) circle (0.05);
\draw [fill] (15:3) circle (0.05);
\draw [fill] (0:3) circle (0.05);
\draw [fill] (325:3) circle (0.05);
\draw [fill] (310:3) circle (0.05);
\draw [fill] (295:3) circle (0.05);
\draw [fill] (280:3) circle (0.05);
\draw [fill] (265:3) circle (0.05);
\draw [fill] (220:3) circle (0.05);
\draw [fill] (235:3) circle (0.05);
\draw [fill] (205:3) circle (0.05);
\draw [fill] (185:3) circle (0.05);
\draw [fill] (170:3) circle (0.05);
\draw [fill] (155:3) circle (0.05);
\draw [fill] (105:3) circle (0.05);

\node  [above] at (90:3) {$x_{0}$};
\node  [above] at (75:3) {$x_{1}$};
\node  [above] at (60:3) {$x_{2}$};
\node  [right] at (325:3) {$x_{j}$};
\node  [below] at (280:3) {$x_{p}$};
\node  [left] at (220:3) {$x_{t}$};

\end{tikzpicture}
}

\newcommand{\sectionThreeTwo}
{
\begin{tikzpicture}[>=stealth']
\draw [thick,dotted](0,0) circle (3);

\draw[->,thick,red,double] (60:3) arc (60:45:3);
\draw[->,thick,red,double] (45:3) -- (0:3);
\draw[->,thick,red,double] (30:3) arc (30:15:3);
\draw[->,thick,red,double] (325:3) arc (325:310:3);
\draw[->,thick,red,double] (295:3) arc (295:280:3);
\draw[->,thick,red,double] (280:3) -- (230:3);

\draw[->,thick,red,double] (230:3) arc (230:215:3);
\draw[->,thick,red,double] (195:3) arc (195:170:3);
\draw[->,thick,red,double] (155:3) arc (155:140:3);
\draw[->,thick,red,double] (105:3) arc (105:90:3);
\draw[->,double,red,double] (170:3)-- (105:3);

\draw[->,thick,] (90:3) arc (90:75:3);
\draw[->,thick,] (75:3) arc (75:60:3);
\draw[->,thick,] (45:3) arc (45:30:3);
\draw[->,thick,] (15:3) arc (15:0:3);
\draw[->,thick,] (310:3) arc (310:295:3);
\draw[->,thick,] (280:3) arc (280:265:3);
\draw[->,thick,] (245:3) arc (245:230:3);
\draw[->,thick,] (170:3) arc (170:155:3);
\draw[->,thick] (90:3) -- (280:3);
\draw[->,thick] (140:3)--(45:3);
\draw[->,thick] (230:3)--(170:3);
\draw[->,thick] (105:3)--(325:3);

\draw [fill] (90:3) circle (0.05);
\draw [fill] (75:3) circle (0.05);
\draw [fill] (60:3) circle (0.05);
\draw [fill] (45:3) circle (0.05);
\draw [fill] (30:3) circle (0.05);
\draw [fill] (15:3) circle (0.05);
\draw [fill] (0:3) circle (0.05);
\draw [fill] (325:3) circle (0.05);
\draw [fill] (310:3) circle (0.05);
\draw [fill] (295:3) circle (0.05);
\draw [fill] (280:3) circle (0.05);
\draw [fill] (265:3) circle (0.05);

\draw [fill] (195:3) circle (0.05);
\draw [fill] (170:3) circle (0.05);
\draw [fill] (155:3) circle (0.05);
\draw [fill] (140:3) circle (0.05);
\draw [fill] (105:3) circle (0.05);

\node  [above] at (90:3) {$x_{0}$};
\node  [above] at (75:3) {$x_{1}$};
\node  [above] at (60:3) {$x_{2}$};
\node  [right] at (325:3) {$x_{p_j}$};
\node  [below] at (280:3) {$x_{p}$};
\node  [left] at (230:3) {$x_{p_1}$};
\node  [left] at (170:3) {$x_{p_2}$};
\node  [left] at (105:3) {$x_{p_{j-1}}$};

\end{tikzpicture}
}
\begin{document}
 
\title{Proper connection and proper-walk connection of digraphs}
\author{
  Anna Fiedorowicz$^1$, El\.zbieta Sidorowicz$^1$ and  \'Eric Sopena$^2$\\
  {\scriptsize $^1$ Institute of Mathematics, University of Zielona G\'ora, Zielona G\'ora, Poland}\\
  {\scriptsize e-mail: a.fiedorowicz@wmie.uz.zgora.pl; e.sidorowicz@wmie.uz.zgora.pl}\\
  {\scriptsize $^2$ Univ. Bordeaux, CNRS, Bordeaux INP, LaBRI, UMR5800, F-33400 Talence, France}\\
  {\scriptsize e-mail: eric.sopena@labri.fr}
   }


\maketitle
\begin{abstract}
An arc-colored digraph $D$  is properly  (properly-walk) connected if, for any ordered pair of vertices $(u,v)$, the digraph $D$ contains a directed path (a directed walk) from $u$ to $v$ such that arcs adjacent on that path (on that walk) have distinct colors. 
The  proper connection number $\pc(D)$ (the proper-walk connection number $\wc(D)$) of a digraph $D$ is the minimum number of colours to make $D$ properly connected (properly-walk  connected). 
We prove that $\pc(C_n(S))\le 2$ for every circulant digraph $C_n(S)$ with $S\subseteq\{1,\ldots ,n-1\}, |S|\ge 2$ and $1\in S$. 
Furthermore, we give some sufficient conditions for a Hamiltonian digraph $D$
to satisfy $\pc(D) = \wc(D) = 2$.
\end{abstract}


\section{Introduction}\label{sec:introduction}

In an edge-coloured graph $G$, a path $P$ is  {\em rainbow} if no two edges in $P$ have  the same colour. 
An edge-coloured graph $G$ is {\it rainbow connected} if every two vertices of $G$ are connected by a rainbow path, and its colouring is said to be a {\it rainbow connected colouring}. 
The {\it rainbow connection number} of $G$ is the smallest possible number of colours in a rainbow connected colouring of $G$. 
The  rainbow connection number of graphs was introduced by Chartrand, Johns, McKeon and Zhang in~\cite{CJMZ08}. 

A weaker version of the rainbow connection number  --the proper connection number-- was introduced by Borozan {\it et al. }in \cite{BoFu12}. 
An edge-coloured graph is said to be {\it properly coloured} if no two adjacent edges share the same  colour. 
A connected edge-coloured graph $G$ is {\it properly connected} if there exists a properly coloured path between every two vertices in $G$. 
The {\it proper connection number} of a connected graph $G$  is the minimum number of colours needed to colour the edges of $G$ to make it properly connected.  

The concepts of rainbow connected graphs and  properly connected graphs have attracted much attention during the last decade. For more details, the reader can refer to surveys~\cite{LS,lishsu-13} (for rainbow connection)  and~\cite{LiMa15,LiMaQi18} (for proper connection).
Melville and  Goddard considered in~\cite{MeGo16,MeGo18} the analogous concept for walks and trails. For a connected graph $G$,  the {\it proper-walk (proper-trail) connection number} is the minimum number of colours that one needs in order to get a properly coloured walk (trail) between every two vertices in $G$.

The rainbow connection, the proper connection and the proper-walk  connection numbers of graphs readily extend to digraphs, using arc-colourings instead of edge-colourings and  directed paths (directed walks, respectively) instead of the paths (walks, respectively). 
The study of rainbow connections of digraphs was initiated by Dorbec {\it et al.} in \cite{DoSc14}.
Then the  rainbow connection number of some digraph classes was determined and different notions similar to the rainbow
connection were introduced, such as the strong rainbow connection, the  rainbow vertex connection and the rainbow total connection 
(see \cite{almo15-1,LeLiLi18,LeLiMa18,SiSo18_1,SiSo18_2}).

The directed version of the proper connection was introduced by Magnant {\it et al.} in \cite{MaMo16}, and the directed version of the proper-walk connection by Melville and  Goddard  in \cite{MeGo16}. 
In \cite{MaMo16} and in \cite{MaNie19} the strong and the vertex version of directed proper connection were considered.  
In this paper we study the proper connection and the proper-walk  connection  of digraphs.

\medskip

An arc-coloured directed path (directed walk) is {\it properly coloured} if  it does not contain two adjacent arcs with the same colour. 
An  arc-coloured  digraph $D$ is {\it properly connected} if, between every ordered pair of vertices, there is a directed properly coloured path. 
In that case, we say that the corresponding arc-colouring is a properly connected arc-colouring of $D$.
The {\it  proper connection number} of  $D$, denoted by $\pc(D)$, is the minimum number of colours needed to colour the arcs of $D$ so that $D$ is properly connected. 

An arc-coloured  digraph $D$ is {\it properly-walk connected} if, between every ordered pair of vertices, there is a  properly coloured directed walk. 
Again, we say that the corresponding arc-colouring is a properly-walk connected arc-colouring of $D$.
Clearly, every properly connected digraph is also properly-walk connected. 
The {\it  proper-walk connection number} of  $D$, denoted by $\wc(D)$, is the minimum number of colours needed to colour the arcs of $D$ so that $D$ is properly-walk connected. 
Note that in order to admit an arc-colouring which makes it properly (properly-walk) connected, a digraph must be strongly connected.

Magnant {\it et al.}~\cite{MaMo16} proved that the  proper connection number of every strong digraph is at most 3. This result suggests the problem of characterizing the digraphs whose proper connection number is at most 2. 
Ducoff {\it et al.}~\cite{DuMa19} proved that determining whether $\pc(D)\le 2$ is NP-complete for any given digraph $D$. 
Gu {\it et al.}~\cite{GuDeLi19} considered the  proper connection number of random digraphs. They proved that if the probability $p$ is at least $(\log n+\log \log n+\lambda(n))/n$, then  the random digraph $D(n,p)$ satisfies $\pc (D(n,p))\le 2$ with high probability. 

The following observation is important for our study.
Given two  digraphs $D_1$ and $D_2$ such that $D_1$ is a spanning subdigraph of $D_2$, if $D_1$ is properly connected under some arc-colouring then, by using that arc-colouring for the arcs of $D_2$, we obtain a partial arc-colouring which makes $D_2$ properly connected. 
Since every even cycle has an arc-colouring with two colours that makes it properly connected,  every Hamiltonian digraph of even order has proper connection number at most~2. 
We will thus focus on Hamiltonian digraphs of odd order. 

The motivation for this paper was the theorem proved by Magnant {\it et al.} in~\cite{MaMo16} which states that the  proper connection number of a strong tournament of order at least~4 is~2. 
On one other hand, the  proper connection number of any odd directed cycle is~3. 
We can thus try to determine the maximum number of arcs we can remove from an odd tournament while keeping  its  proper connection number equal to~2. 
Since a digraph $D$ must be strongly connected  to be properly connected, we must assume that,  when removing arcs from $D$, the resultant digraph is still  strongly connected. 
In this paper  we will assume that the resulting digraph has a directed cycle going through all its vertices, i.e., that the resultant digraph is Hamiltonian.  

\medskip

Our paper is organised as follows.
We introduce definitions and notation in Section~\ref{sec:definitions} and give some preliminary results in Section~\ref{sec:preliminaries}.
In Section~\ref{sec:path}, we prove that if $D$ is a Hamiltonian digraph such that either (i) for every vertex $v$  there is an even chord  with a tail in $v$, or (ii) for every vertex $v$ there is an even chord  with a head in $v$, then the  proper connection number of $D$ is at most~2. 
We also prove that $\pc(C_n(S))\le 2$ for every circulant digraph $C_n(S)$ with $S\subseteq\{1,\ldots ,n-1\},|S|\ge 2$ and $1\in S$. 
This result implies a theorem proved in~\cite{MaNie19}, which states that $\pc(C_n([k])=2$ whenever $k\neq n-1$ and $k\neq 1$. 
In Section~\ref{sec:walk}, we give some sufficient conditions for a Hamiltonian digraph to have  proper-walk connection number  equal to~2. 
We conclude the paper with some open questions in Section~\ref{sec:remarks}.


\section{Definitions and notation}\label{sec:definitions}

All digraphs in this paper are simple in the following sense: they are loopless, they do not contain parallel arcs, but opposite arcs are allowed.
For a given digraph $D$, we denote by $V(D)$ and $A(D)$  its set of vertices and its set of arcs, respectively.
Two arcs $xy$ and $zt$ in $D$ are said to be  {\em consecutive} if $y=z$.
Given an arc $xy$ in $D$, we say that $y$ is an {\em out-neighbour} of $x$, while $x$ is an {\em in-neighbour} of $y$. 
Moreover, $x$ is the {\em tail} of $xy$ and $y$ the {\em head} of $xy$.
The {\em out-degree} $d_D^+(x)$ of $x$ in $D$ is the number of arcs with the tail in $x$ and  the {\em in-degree} $d_D^-(x)$ of $x$ in $D$ is the number of arcs with the head in $x$. The {\em degree} $d_D(x)$ of $x$ in $D$ is the number of arcs incident with $x$, $d_D(x)=d_D^+(x)+d_D^-(x)$.

Let $D$ be a digraph. For an arc $xy$ in $A(D)$, we denote by $D-xy$ the digraph $D-xy=(V(D),A(D)\setminus\{xy\})$.
For a vertex $u$ in $V(D)$, we denote by $D-u$ the digraph $D-u=(V(D)\setminus\{u\},(A(D)\setminus(\{u\}\times V(D))\setminus(V(D)\times\{u\}))$.
For a digraph $D'$, we denote by $D\cup D'$ the digraph 
$D\cup D'=(V(D)\cup V(D'),A(D)\cup A(D'))$.

For a digraph $D$ and a set of vertices $S\subseteq V(D)$, we denote by $D[S]$ the subdigraph of $D$ induced by $S$, that is, $D[S]=(S,A(D)\cap (S\times S))$.

A {\it walk} of length $k\ge 1$ in a digraph $D$ is a sequence $x_0\dots x_k$ of vertices 
such that $x_ix_{i+1}\in A(D)$ for every $i$, $0\le i\le k-1$. 
A {\it path}  is a walk in which  no vertex appears twice. 
Such a path (a walk), going from $x_0$ to $x_k$, is referred to as an {\it $x_0x_k$-path} (an {\it $x_0x_k$-walk}). Let $P=x_0\dots x_k$ be an $x_0x_k$-path and $xx_0\in A(D)$ ($x_kx'\in A(D)$). By $xx_0P$ ($Px_kx'$) we denote the path $xx_0\dots x_k$ ($x_0\dots x_kx'$).

An {\it ear} in a digraph $D$ is an $xy$-path $Q$ such that
$d_D(x)>2$, $d_D(y)>2$ and $d_D(z)=2$ for every internal vertex $z$ of $Q$.

A digraph $D$ is {\it strongly connected} ({\it strong}, for short) if, for every ordered pair of vertices $(u,v)$, there exists a  $uv$-path in $D$.

A {\it cycle} of length $k\ge 1$ in a digraph $D$ is a sequence $x_0\dots x_kx_0$ of vertices such that $x_ix_{i+1}\in A(D)$ for every $i$, $0\le i\le k$  (subscripts are taken modulo $k$), and $x_i\neq x_j$ for every $i,j$, $0\le i<j\le k$.  
Since cycles are denoted similarly all along the paper, 
it is taken for granted that subscripts are always taken modulo the length of the cycle. 

For a path $P$, we denote by $|P|$ the length of $P$, i.e., $|P|=|A(P)|$. 
For a cycle $C$, we denote by $C[x_i,x_j]$ the $x_ix_j$-path contained in $C$ (that is, $C[x_i,x_j]=x_ix_{i+1}\dots x_j$).

The {\em distance} from a vertex $x$ to a vertex $y$ in a digraph $D$,
denoted by $\dist_D(x,y)$, is the length of a shortest $xy$-path in $D$
(if there is no such path, we let $\dist_D(x,y)=\infty$).

Let $C=x_0\dots x_{n-1}x_0$ be a cycle of length $n$.
A {\em chord} of $C$ is an arc $x_px_q$, $0\le p,q\le n-1$, with $x_q\neq x_{p+1}$. 
The {\it length} of the chord $x_px_q$ is $|C[x_p,x_q]|$. 
The chord $x_px_q$ is {\it even}  if $|C[x_p,x_q]|$ is even, otherwise, the chord $x_px_q$ is {\it odd}.

Let $D$ be a Hamiltonian digraph, $C$ be a Hamiltonian cycle of $D$, and $u,v$ be two adjacent vertices of $D$. 
If $uv$ is an arc of $C$ then we write $u\ri v$; if $uv$ is a chord of $C$, then we write $u\Ri v$. 
For a path $x_0x_1\dots x_{k-1}x_k$, we write $x_0\Rri x_1\Rri \dots \Rri x_{k-1}\Rri x_k$, where $\Rri$ stands either for $\ri$ or $\Ri$, depending on the type of the corresponding arc.

Given a graph $G=(V,E)$, the {\it biorientation} of $G$ is the symmetric digraph $\overleftrightarrow{G}$ obtained from $G$ by replacing each edge $uv$ of $G$ by the pair of symmetric arcs $uv$ and $vu$.

For an integer $n\ge 3$ and a set $S\subseteq \{1,2,\ldots, n-1\}$, the {\it circulant digraph} $C_n(S)$ is the digraph with vertex set $V(C_n(S))=\{v_0, v_1, \ldots , v_{n-1}\}$ and arc set $A(C_n(S))=\{v_iv_j : j-i\equiv s \pmod n, s\in S\}$.

\section{Preliminaries}\label{sec:preliminaries}

For every digraph $D$, we denote by $D^{-1}$ the \emph{reversed digraph of $D$}, that is, $V(D^{-1})=V(D)$ and $xy\in A(D^{-1})$ if and only if $yx\in A(D)$.
It directly follows from the definitions that if an arc-colouring $\lambda$ makes $D$ properly connected (properly-walk connected), then the colouring $\lambda'$ defined by $\lambda'(xy)=\lambda(yx)$ for every arc $xy\in A(D^{-1})$ makes $D^{-1}$ properly connected (properly-walk connected). Hence, we have the following observation.

\begin{obs}\label{obs:reversed-digraph}
For every digraph $D$, $\pc(D^{-1})=\pc(D)$ and $\wc(D^{-1})=\wc(D)$.
\end{obs}

Magnant {\it et al.}  \cite{MaMo16} proved the following theorem.

\begin{theorem}{\rm \cite{MaMo16}}
\label{thm:strong digraph}
If $D$ is a strong digraph, then $\pc(D)\le 3$.
\end{theorem}

The only digraphs with  proper connection number equal to~1 are biorientations of  complete graphs $\overleftrightarrow{K_n}$.  
Since it is NP-complete to decide whether a strong digraph has proper connection number at most~2 \cite{DuMa19}, it seems to be interesting to find some sufficient conditions for a digraph to have this property. 

One can see that every bipartite strong digraph, except $\overleftrightarrow{K_2}$,  has   proper connection number equal to~2. Indeed, let $D=(X\cup Y, A)$ be a strong bipartite digraph. 
If we colour   all arcs with tail in $X$ with colour~1 and  all arcs with tail in $Y$ with colour~2, then we clearly obtain a properly connected digraph. 

\begin{prop}
If $D$ is a strong bipartite digraph, $D\neq \overleftrightarrow{K_2}$, 
then $\pc(D) = 2$.
\end{prop}

If $D$ has a spanning subdigraph $H$ with $\pc(H)\le 2$, then $\pc(D)\le 2$.  We thus obtain the following corollary.

\begin{corollary}
\label{cor:bipartite}
If a digraph $D$ has a strong bipartite spanning subdigraph, then $\pc(D)\le 2$.
\end{corollary}

In the next theorem we give another family of digraphs having  proper connection number at most~2. 
Consequently, every digraph which contains a spanning digraph belonging to this family also has  proper connection number at most~2.

\begin{theorem}
\label{thm:almost_bipartite}
Let $D$ be a strong digraph. If  there is a partition $V(D)=V_1\cup V_2$,
with $V_1\cap V_2=\emptyset$, such that $D[V_1]$ is  strong bipartite  and $V_2$ is  an independent set of vertices, then $\pc(D)\le 2$. 
\end{theorem}

\begin{proof}
We first consider the subdigraph $D[V_1]$. 
Let $V_1=V'_1\cup V''_1$ be the partition of $V_1$ into two independent sets. 
We then colour the arcs with tail in $V'_1$ with~1, and the arcs with tail in $V''_1$   with~2. Observe that every path linking two vertices in $D[V_1]$ is properly coloured
and, moreover, the first colour of every path starting from $V'_1$ (resp. from $V''_1$) is~1 (resp. 2) and the last colour of every path ending in $V'_1$ (resp. in $V''_1$) is~2 (resp. 1).

We now colour the remaining arcs of $D$ as follows. 
Let $xy$ be an arc with $x\in V_1$ and $y\in V_2$. We then colour $xy$ with colour~1 if $x\in V'_1$ and with colour~2 if $x\in V''_1$.
On the other hand, if $xy$ is an arc with $x\in V_2$ and $y\in V_1$, then we colour $xy$ with colour~1 if $y\in V''_1$ and with colour~2 if $y\in V'_1$.
Observe that every directed path of length~2 starting from $V_2$ or ending in $V_2$ is properly coloured.

We now claim that the so-constructed arc-colouring makes $D$ properly connected.
Let $(x,y)$ be any ordered pair of distinct vertices in $D$. Let $x'$ be any out-neighbour of $x$ in $V_1$ if $x\in V_2$ (such a vertex necessarily exists since $D$ is strong), and $x'=x$ otherwise.
Similarly, let $y'$ be any in-neighbour of $y$ in $V_1$ if $y\in V_2$ (again, such a vertex necessarily exists since $D$ is strong), and $y'=y$ otherwise.
From the above observations, we get that for any path $P$ in $D[V_1]$ from $x'$ to $y'$ (with possibly $x'=y'$, in which case $P$ is the empty path), the path $xPy$ is necessarily properly coloured.
Therefore, the so-constructed arc-colouring makes $D$ properly connected, and thus $\pc(D)\le 2$. 
\end{proof}

\begin{corollary}
\label{cor:almost_bipartite}
If a digraph $D$ has a spanning subdigraph satisfying the conditions of Theorem~\ref{thm:almost_bipartite}, then $\pc(D)\le 2$.
\end{corollary}

\section{Proper connection of digraphs}\label{sec:path}

In this section we consider the proper connection number of Hamiltonian digraphs. Observe that if a Hamiltonian digraph $D$ has even order, then it has a spanning even cycle, which has  proper connection number~2, and it follows that $\pc(D)\le 2$. We thus focus on Hamiltonian digraphs of odd order. First we consider Hamiltonian digraphs having even chords.

\begin{theorem}
\label{thm:even_tail}
Let $D$ be a Hamiltonian digraph of odd order and $C$ be a Hamiltonian cycle of $D$. If every vertex of $D$ is the tail of an even chord of $C$, 
then $\pc(D)\le 2$.
\end{theorem}

\begin{proof}
We will show that $D$ contains a strong bipartite spanning subdigraph. 
Let $C=x_0\dots x_{n-1}x_0$.  Renaming vertices if necessary, we may assume that $x_px_0$, for some~$p$, $1\le p\le n-2$,  is an even chord of $C$ with the minimum value of $|C[x_p,x_0]|$.

Let $V_1\cup V_2$ be the partition of $V(D)$ given by $V_1=\{x_{2i}:i\in\{0,\ldots,(n-1)/2\}\}$ and $V_2=\{x_{2i+1}:i\in\{0,\ldots,(n-3)/2\}\}$. 
We claim that there is a strong bipartite spanning subdigraph of $D$ with vertex partition $(V_1,V_2)$. 
Observe that every arc of $A(C)\setminus \{x_{n-1}x_0\}$ joins two vertices from  different sets of the partition. Furthermore, every even chord $x_ix_j$ such that $x_{n-1}x_0\subseteq C[x_i,x_j]$ also joins two vertices from  different sets of the partition. 
Our initial  assumption that $|C[x_p,x_0]|$ is minimum implies that, for every even chord $x_rx_s$ with  $x_r\in C[x_{p+1},x_{n-1}]$, we have $x_s\in C[x_1,x_{r-1}]$, and thus $x_{n-1}x_0\subseteq C[x_r,x_s]$. 
Therefore, every such chord joins vertices of two different sets of the partition. 

Let $A'$ be the set of even chords with  tail in $C[x_{p+1},x_{n-1}]$.  Then the  spanning subdigraph $D^*$, with set of arcs $A(D^*)=A(C)\setminus \{x_{n-1}x_0\}\cup\{x_px_0\}\cup A'$ , is bipartite. 
We finally claim that $D^*$ is a strong subdigraph.
Note that $x_0x_1\dots x_px_0$ is a cycle and $x_px_{p+1}\dots x_{n-1}$ a path in $D^*$.
It thus suffices to prove that there is a path from any vertex in $C[x_{p+1},x_{n-1}]$ to $v_p$. Indeed, this follows from the fact that for every vertex $x_r$ in $C[x_{p+1},x_{n-1}]$, there is an arc $x_rx_s$ in $A'$ with $s<r$.
\end{proof}

Thanks to Observation~\ref{obs:reversed-digraph}, and the fact that an even chord of some cycle $C$ in a digraph~$D$ is an even chord of the ``reversed cycle'' $C^{-1}$ in $D^{-1}$, we get the following corollary.

\begin{corollary}
\label{cor:even_head}
Let $D$ be a Hamiltonian digraph of odd order and $C$ be a Hamiltonian cycle of $D$. If  every vertex of $D$ is the head of an even chord of $C$,
then $\pc(D)\le 2$.
\end{corollary}

Now, from Theorem~\ref{thm:even_tail} and Corollary~\ref{cor:even_head}, we get the following corollary.

\begin{corollary}
Let $D$ be a Hamiltonian digraph and $C$ be a Hamiltonian cycle of $D$. If $C$ contains only even chords 
and $d_D^-(v)\ge 2$  for every $v\in V(D)$ (or $d_D^+(v)\ge 2$  for every $v\in V(D)$), then $\pc(D)= 2$. 
\end{corollary}

In the rest of this section we consider circulant digraphs and  prove the following theorem.

\begin{theorem}\label{thm:main}
If $n\ge 4$, $S\subseteq \{1,\ldots, n-1\}$, $|S|\ge 2$, and $1\in S$, then $\pc(C_n(S))\le 2$.
\end{theorem}

Observe first that Theorem~\ref{thm:main}  obviously holds for $n$ even since $C_n(S)$ is Hamiltonian. 
When $n$ is odd, Theorem~\ref{thm:even_tail} and Corollary~\ref{cor:even_head} give the following result.

\begin{corollary}
\label{cor:even-integer}
If $n$ is odd,  $1\in S$, and $S$ contains an even integer, then $\pc(C_n(S))\le 2$.
\end{corollary}

Thus to complete the proof of Theorem~\ref{thm:main} it is enough to show that for every two odd integers $n$ and $k$,  $n\ge 5$, $3\le k<n$, we have $\pc(C_n(\{1,k\}))= 2$. 
To prove this result we need the following lemma, which has been proved in \cite{DuMa19}, using a construction that we recall in Figure~\ref{fig:C_seven_three}. 

\begin{lemma}{\rm \cite{DuMa19}}
\label{lem:C_seven_three}
$\pc(C_7(\{1,3\}))=2$.
\end{lemma}

\begin{figure}
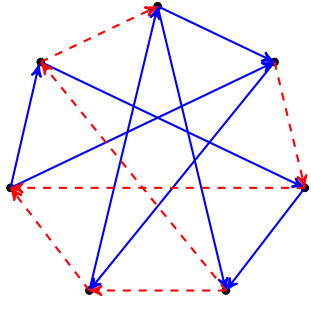

\begin{center}
\CSevenThree
\end{center}
\caption{A properly connected 2-arc-colouring of $C_7(\{1,3\})$.}
\label{fig:C_seven_three}
\end{figure}

Observe that $C_7(\{1,3\})$ and $C_7(\{1,5\})$ are isomorphic graphs, so that we also have $\pc(C_7(\{1,5\}))=2$. 
We are now able to deal with the remaining cases of Theorem~\ref{thm:main}.

\begin{lemma}
\label{lem:odd}
If $n$ and $k$ are odd,  $n\ge 5$ and $3\le k\le n-1$, then $\pc(C_n(\{1,k\}))= 2$.
\end{lemma}

\begin{proof}
For $n=5$, there is  exactly one  circulant digraph with odd $k$, namely $C_5(\{1,3\})$. Observe that  $C_5(\{1,3\})$ is a strong tournament. Since the  proper connection number of every strong tournament with at least four vertices  is~2 (see~\cite{MaMo16}), 
we get $\pc(C_5(\{1,3\}))\le 2$ and the lemma holds in this case. 

For $n=7$, the result follows from Lemma~\ref{lem:C_seven_three} and the fact that $C_7(\{1,3\})$ and $C_7(\{1,5\})$ are isomorphic. 
Thus we may assume $n\ge 9$. 
We will show that, in that case, the digraph $C_n(\{1,k\})$ contains a strong bipartite spanning subdigraph, which will imply the desired result. 
This subdigraph will be constructed step by step, starting from an even cycle, and adding ears in such a way that the so-obtained subdigraph is still bipartite, until we get a spanning subdigraph. Doing so, the constructed subdigraph will clearly be strong.

The bipartition of the subdigraph will be given by means of a vertex 2-colouring~$c$, simply referred to as a 2-colouring in the rest of the proof.

Let $V(C_n(\{1,k\}))=\{x_0,\dots ,x_{n-1}\}$ and $n=\alpha k+r$, 
with $0\le r\le k-1$ and $\alpha\in\mathbb{N}$. 
We consider three cases, depending on the value of $k$.

\medskip
\noindent {\bf Case 1. $3k-2\le n.$}

We consider three subcases, depending on the value of $r$.

\medskip
\noindent {\bf Subcase  1.1. $r\le k-3.$}

\begin{figure}
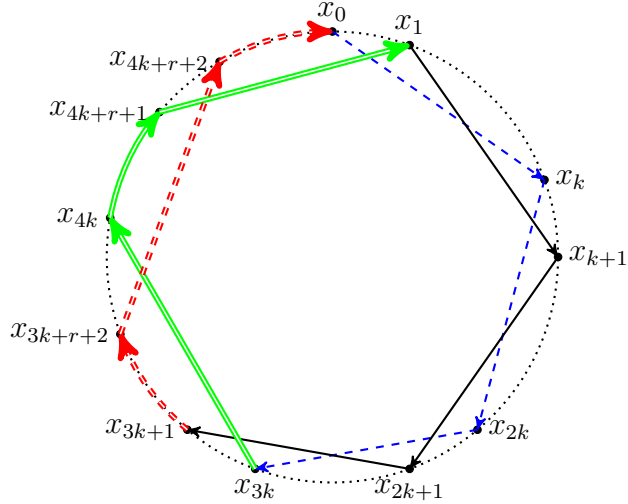

\begin{center}
\caseOneOne
\end{center}
\caption{(Subcase 1.1) A sample digraph with $\alpha=5$, together with the paths
$P_1$ (dashed, blue), $P_2$ (black), $P_3$ (double, green) and $P_4$ (double dashed, red).}
\label{fig:case_one_one}
\end{figure}

In that case, we necessarily have $\alpha\ge 3$.
Consider the four following paths (see Figure~\ref{fig:case_one_one}):
\begin{align*}
P_1&=x_0\Ri x_k\Ri x_{2k} \Ri \ldots \Ri x_{(\alpha -2)k},\\
P_2&=x_1\Ri x_{k+1}\Ri x_{2k+1}\Ri \ldots \Ri x_{(\alpha -2)k+1},\\
P_3&=x_{(\alpha-2)k}\Ri  x_{(\alpha-1)k}\ri  \ldots \ri x_{(\alpha-1)k+r+1}\Ri x_1,\\
P_4&=x_{(\alpha-2)k+1}\ri  \ldots \ri x_{(\alpha-2)k+r+2}\Ri   x_{(\alpha-1)k+r+2}\ri  \ldots \ri x_0.
\end{align*}

Observe that 
$C=P_1\cup P_2\cup P_3\cup P_4$ is a cycle. 
Furthermore, since $|P_1|=|P_2|=\alpha-2$, $|P_3|= r+3$ and $|P_4|=k+r$, we get
$|A(C)|=2\alpha+2r+k-1$, and thus $C$ is an even cycle. 
Let $c$ be a 2-colouring  of $C$ with $c(x_1)=1$. 

The path $P_2\cup P_4$ goes from $x_1$ to $x_0$ in $C$, thus $d_C(x_1,x_0)=\alpha+k+r-2$. Since $n=\alpha k+r$ is odd and $k$ is odd, $\alpha $ and $r$ are of different parity. Hence, $d_C(x_1,x_0)$ is even and thus $c(x_0)=c(x_1)=1$. 
More generally, we have 
$c(x_{ik})=c(x_{ik+1})=1$ for $i\in \{0,\ldots ,\alpha -2\}$, $i$ even, 
and $c(x_{ik})=c(x_{ik+1})=2$ for $i\in \{0,\ldots ,\alpha -2\}$, $i$ odd. 

We are now ready to add ears to the cycle $C$ in order to get a spanning subdigraph of $C_n(\{1,k\})$. 
We consider the following paths of $C_n(\{1,k\})$:
\begin{align*}
Q_i&=x_{ik+2}\ri x_{ik+3}\ri\ldots \ri x_{(i+1)k-1}, \text{ for every } i\in\{0,\ldots, \alpha -3\},\\
Q&=x_{(\alpha-2)k+r+3}\ri x_{(\alpha-2)k+r+4}\ri \ldots \ri x_{(\alpha -1)k-1}.
\end{align*}

Observe  that $V(C_n(\{1,k\}))=V(C)\cup \bigcup_{i\in\{0,\ldots, \alpha -3\}}V(Q_i)\cup V(Q)$. Moreover, since $k$ is odd, the length of each path $Q_i$, $0\le i\le \alpha -3$, is even.

We will add to the cycle $C$ an ear $E_0$ containing $Q_0$ and then, sequentially,
we will add to $C\cup E_0\cup\dots\cup E_{i-1}$ an ear $E_i$ containing $Q_i$, for each $i$, $1\le i\le \alpha-3$.
(We will add later an ear $E$ containing $Q$ to $C\cup E_0\cup\dots\cup E_{\alpha-3}$).
Since each path $Q_i$ has even length, the ends of each ear $E_i$ must have the same colour.
We proceed as follows.

\begin{itemize}
\item We first let $E_0=x_{(\alpha-1)k+r+2} \Ri Q_0 \ri x_k$.
As observed above, we have $c(x_k)=2$. 
Moreover, since $d_C(x_{(\alpha-1)k+r+2},x_0)=k-2$ and $c(x_0)=1$, we also have $c(x_{(\alpha-1)k+r+2})=2$. 
Hence, $C\cup E_0$ is a strong bipartite subdigraph of $C_n(\{1,k\})$ to which we can extend the 2-colouring $c$. In particular, we have $c(x_2)=1$.

\item We now let $E_1 = x_2\Ri Q_1\ri x_{2k}$. Since $c(x_2)=c(x_{2k})=1$, $C\cup E_0\cup E_1$ is a strong bipartite subdigraph of $C_n(\{1,k\})$ to which we can extend the 2-colouring $c$.
Note here that we have $c(x_{k+2})=2$. 

\item Assume finally that we have added ears $E_0,\dots, E_{i-1}$, for some $i<\alpha-3$, in such a way that $C\cup E_0\cup\dots\cup E_{i-1}$ is a strong bipartite subdigraph of $C_n(\{1,k\})$ and the corresponding 2-colouring $c$ is such that
$c(x_{jk+2})=1$ if $j$ is even and $c(x_{jk+2})=2$ if $j$ is odd, for every $j\in\{0,\ldots i-1\}$. 

We then let $E_i=x_{(i-1)k+2}\Ri Q_i\ri x_{(i+1)k}$. Since $c(x_{(i-1)k+2})=c(x_{(i+1)k})=1$ if $i$ is odd, and $c(x_{(i-1)k+2})=c(x_{(i+1)k})=2$ if $i$ is even, 
$C\cup E_0\cup\dots\cup E_{i}$ is a strong bipartite subdigraph of $C_n(\{1,k\})$
to which we can extend the 2-colouring $c$.
Note that we then have $c(x_{ik+2})=1$ if $i$ is even and $c(x_{ik+2})=1$ if $i$ is odd.

\end{itemize}

We finally let $E=x_{(\alpha-2)k+r+2}\ri Q\Ri x_{\alpha k-1}$. 
In order to obtain a bipartite subdigraph, we need to have 
$c(x_{(\alpha-2)k+r+2})=c(x_{\alpha k-1})$ if $|E|$ is even,
and $c(x_{(\alpha-2)k+r+2}) \neq c(x_{\alpha k-1})$ if $|E|$ is odd.
To prove that this holds, it is enough to show that 
$|E|$ and  $d_C(x_{(\alpha-2)k+r+2},x_{\alpha k-1})$ have the same parity.
Indeed, we have $|E|=k-r-2$ and 
\begin{align*}
d_C(x_{(\alpha-2)k+r+2},x_{\alpha k-1}) &= d_{P_4}(x_{(\alpha-2)k+r+2},x_{\alpha k-1})\\
&=1+d_{P_4}(x_{(\alpha-1)k+r+2},x_{\alpha k-1})=k-r-2,
\end{align*}
which completes this subcase.

\begin{figure}
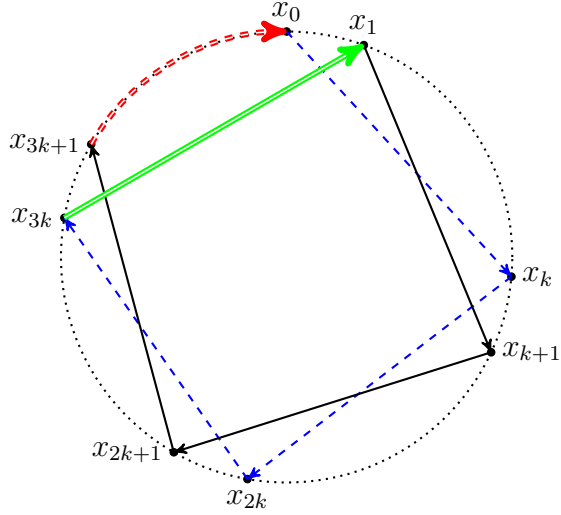

\begin{center}
\caseOneTwo
\end{center}
\caption{(Subcase 1.2) A sample digraph with $\alpha=3$, together with the paths
$P_1$ (dashed, blue), $P_2$ (black), $P_3$ (double, green) and $P_4$ (double dashed, red).}
\label{fig:case_one_two}
\end{figure}

\medskip
\noindent {\bf Subcase 1.2. $r=k-1.$} 

In that case, we have $n=\alpha k+k-1$. Since $n$ and $k$ are odd, it follows that $\alpha $ is odd and thus $\alpha\ge 3$. Consider the four following paths (see Figure~\ref{fig:case_one_two}):
\begin{align*}
P_1&=x_0\Ri x_k\Ri  \ldots \Ri x_{\alpha k},\\
P_2&=x_1\Ri x_{k+1}\Ri  \ldots \Ri x_{\alpha k+1},\\
P_3&=x_{\alpha k}\Ri x_1,\\
P_4&=x_{\alpha k+1}\ri x_{\alpha k+2}\ri  \ldots  \ri x_{\alpha k+k-2}\ri x_0.
\end{align*}

Observe that   $C=P_1\cup P_2\cup P_3\cup P_4$ is a cycle. 
Furthermore, since $|P_1|=|P_2|=\alpha, |P_3|= 1$ and $|P_4|=k-2$, we get $|A(C)|=2\alpha+k-1$, and thus $C$ is an even cycle. 
Let $c$ be a 2-colouring  of $C$ with $c(x_1)=1$. 

The path $P_1\cup P_3$ goes from $x_0$ to $x_1$ in $C$, thus $d_C(x_0,x_1)=|P_1\cup P_3|=\alpha+1$. 
This implies $c(x_1)=c(x_0)$ and thus 
$c(x_{ik})=c(x_{ik+1})=1$ for $i\in \{0,\ldots ,\alpha \}$, $i$ even, 
and $c(x_{ik})=c(x_{ik+1})=2$ for $i\in \{0,\ldots ,\alpha \}$, $i$ odd. 

We now consider the following paths of $C_n(\{1,k\})$:
\begin{align*}
Q_i=x_{ik+2}\ri x_{ik+3}\ri \dots \ri x_{(i+1)k-1} \text{ for every } i\in\{0,\ldots, \alpha -1\}.
\end{align*}

Observe  that $V(C_n(\{1,k\}))=V(C)\cup \bigcup_{i\in\{0,\ldots, \alpha -1\}}V(Q_i)$. Similarly as before, we will add to the cycle $C$ an ear $E_0$ containing $Q_0$ and then, sequentially,
we will add to $C\cup E_0\cup\dots\cup E_{i-1}$ an ear $E_i$ containing $Q_i$, for each $i$, $1\le i\le \alpha-1$. 
Since every path $Q_i$ has even length, the ends of each ear $E_i$ must have the same colour.
We proceed as follows.

\begin{itemize}
\item We first let $E_0=x_{\alpha k+1}\Ri Q_0\ri x_k$. 
Since $d_C(x_{\alpha k+1},x_{k})=|P_4\Ri x_k|=k-1$ is even, we have $c(x_{\alpha k+1})=c(x_{k})$. 
Hence, $C\cup E_0$ is a strong bipartite subdigraph of $C_n(\{1,k\})$ to which we can extend the 2-colouring $c$. In particular, since $c(x_{\alpha k+1})=2$ and $x_{\alpha k+1}x_2\in A(C\cup E'_0)$, we have $c(x_2)=1$.

\item Assume now that we have added ears $E_0,\dots E_{i-1}$, for some $i<\alpha-1$,
in such a way that $C\cup E_0\cup\dots\cup E_{i-1}$ is a strong bipartite subdigraph of $C_n(\{1,k\})$,
and the corresponding 2-colouring $c$ is such that $c(x_{jk+2})=1$ if $j$ is even and $c(x_{jk+2})=2$ if $j$ is odd, for every $j\in\{0,\ldots i-1\}$. 

We then let $E_i=x_{(i-1)k+2}\Ri E_i\ri x_{(i+1)k}$. 
Since $c(x_{(i-1)k+2})=c(x_{(i+1)k})=1$ if $i$ is odd and $c(x_{(i-1)k+2})=c(x_{(i+1)k})=2$ if $i$ is even, 
$C\cup E_0\cup\dots\cup E_{i}$ is a strong bipartite subdigraph of $C_n(\{1,k\})$ to which we can extend the 2-colouring $c$.
Note that we then have $c(x_{ik+2})=1$ if $i$ is even and $c(x_{ik+2})=2$ if $i$ is odd. 
\end{itemize}

The so-obtained subdigraph $C\cup E_0\cup\dots\cup E_{i\alpha-1}$
is a strong bipartite spanning subdigraph of $C_n(\{1,k\})$, which completes this subcase.

\medskip
\noindent {\bf Subcase 1.3. $r=k-2.$} 

In this case, we have $n=\alpha k+k-2$. 
Since $n$ and $k$ are odd, it follows that $\alpha$ is even. 
We consider two subcases, depending on the value of $\alpha$.

\medskip
\noindent {\it Subcase 1.3.1. $\alpha \ge 4$ ($\alpha$ even).} 

\begin{figure}
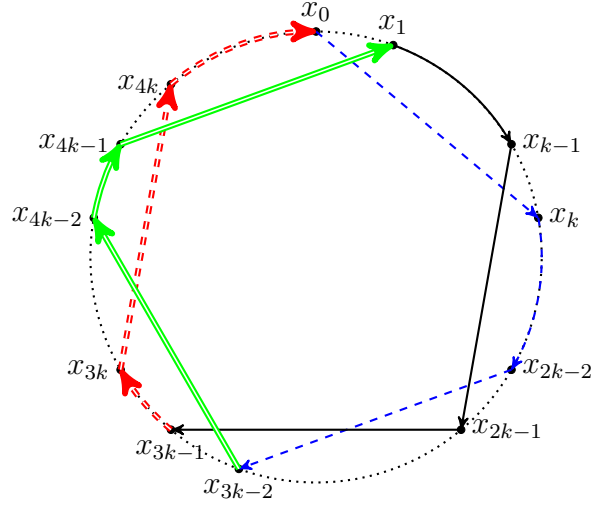

\begin{center}
\caseOneThreeOne
\end{center}
\caption{(Subcase 1.3.1) A sample digraph with $\alpha=4$, together with the paths
$P_1$ (dashed, blue), $P_2$ (black), $P_3$ (double, green) and $P_4$ (double dashed, red).}
\label{fig:case_one_three_one}
\end{figure}

Consider the four following paths (see Figure~\ref{fig:case_one_three_one}):
\begin{align*}
P_1&=x_0\Ri x_k\ri  \dots \ri x_{2k-2}\Ri x_{3k-2} \Ri \dots \Ri x_{(\alpha -1)k-2},\\
P_2&=x_1\ri  x_2\ri  \dots \ri x_{k-1}\Ri x_{2k-1}\Ri x_{3k-1}\Ri \dots \Ri x_{(\alpha-1) k-1},\\
P_3&=x_{(\alpha -1) k-2}\Ri x_{\alpha k-2}\ri x_{\alpha k-1}\Ri x_1,\\
P_4&=x_{(\alpha -1) k-1}\ri x_{(\alpha -1) k}\Ri x_{\alpha k}\ri x_{\alpha k+1}\ri  \dots  \ri x_{\alpha k+k-3}\ri x_0.
\end{align*}

Observe that   $C=P_1\cup P_2\cup P_3\cup P_4$ is a cycle. Furthermore, since $|P_1|=|P_2|=\alpha +k-4, |P_3|= 3$ and $|P_4|=k$, we get $|A(C)|=2\alpha+3k-5$, and thus $C$ is  an even cycle. 
Let $c$ be a 2-colouring  of $C$ with $c(x_0)=1$. 

We first claim that $c(x_{ik-2})=c(x_{ik-1})=1$ if $i$ is even  and 
$c(x_{ik-2})=c(x_{ik-1})=2$ if $i$ is odd for every $i\in \{2,\ldots ,\alpha -1\}$.
Indeed, the path $P_1\cup P_3$ goes from $x_0$ to $x_1$ in $C$, 
and thus $d_C(x_0,x_1)=|P_1\cup P_3|=\alpha+k-1$. 
Since $\alpha$ is even, $d_C(x_0,x_1)$ is even and thus $c(x_0)=c(x_1)=1$. 
Since  $c(x_0)=1$ and $x_0x_k\in A(P_1)$, we have $c(x_k)=2$. 
Since $c(x_1)=1$ and $d_{P_2}(x_1,x_{k-1})=k-2$ is odd, we have $c(x_{k-1})=2$. 
Since $c(x_{k-1})=2$ and $x_{k-1}x_{2k-1}\in A(P_2)$, we have $c(x_{2k-1})=1$.
Since $c(x_k)=2$ and $d_{P_1}(x_k,x_{2k-2})=k-2$ is odd, we have $c(x_{2k-2})=1$. 
Eventually, we get
$c(x_{ik-2})=c(x_{ik-1})=1$ if $i$ is even  and 
$c(x_{ik-2})=c(x_{ik-1})=2$ if $i$ is odd, for every $i\in \{2,\ldots ,\alpha -1\}$. 

We now consider the following paths of $C_n(\{1,k\})$:
\begin{align*}
Q_i &= x_{ik}\ri x_{ik+1}\ri \dots \ri x_{(i+1)k-3} \text{ for every } i\in\{2,\ldots, \alpha -2\},\\
Q &= x_{(\alpha-1)k+1}\ri x_{(\alpha-1)k+2}\ri \dots \ri x_{\alpha k-3}.
\end{align*}

Observe  that $V(C_n(\{1,k\}))=V(C)\cup \bigcup_{i\in\{2,\ldots, \alpha -2\}}V(Q_i)\cup V(Q)$. 
Again, we will add to the cycle $C$ an ear $E_2$ containing $Q_2$ and then, sequentially,
we will add to $C\cup E_2\cup\dots\cup E_{i-1}$ an ear $E_i$ containing $Q_i$, for each $i$, $3\le i\le \alpha-2$. 
(We will add later an ear $E$ containing $Q$ to $C\cup E_2\cup\dots\cup E_{\alpha-2}$).
Since every path $Q_i$ has even length, the ends of each ear $E_i$ must have the same colour.
We proceed as follows.

\begin{itemize}
\item We first let $E_2=x_{k}\Ri Q_2\ri x_{3k-2}$. 
Since $c(x_{k})=2$ and $c(x_{3k-2})=2$, $C\cup E_2$ is a strong bipartite subdigraph of $C_n(\{1,k\})$ and we can extend the 2-colouring $c$. 
In particular, since $c(x_k)=2$ and $x_kx_{2k}$ is an arc in $E_2$, we have $c(x_{2k})=1$.

\item Assume now that we have added ears $E_2,\ldots, E_{i-1}$ for some $i<\alpha-2$,
in such a way that $C\cup E_2\cup\dots\cup E_{i-1}$ is a strong bipartite subdigraph of $C_n(\{1,k\})$, and the corresponding 2-colouring $c$ is such that 
$c(x_{jk})=1$ if $j$ is even and $c(x_{jk})=2$ if $j$ is odd, for every $j\in\{2,\ldots i-1\}$. 

We then let $E_i=x_{(i-1)k}\Ri Q_i\ri x_{(i+1)k-2}$. 
Since $c(x_{(i-1)k})=c(x_{(i-1)k-1})=c(x_{(i-1)k-2})$, we have  $c(x_{(i-1)k})=c(x_{(i+1)k-2})$. 
Therefore, $C\cup E_2\cup\dots\cup E_{i-1}\cup E_i$ is a strong bipartite subdigraph of $C_n(\{1,k\})$ to which we can extend the 2-colouring~$c$.
Note that we have $c(x_{ik})=1$ if $i$ is even and $c(x_{ik})=2$ if $i$ is odd. 
\end{itemize}

We finally let $E=x_{(\alpha-2)k+1}\Ri Q\ri x_{\alpha k-2}$. 
Since the path $Q$ is of odd length, we need to have $c(x_{(\alpha-2)k+1})\neq c(x_{\alpha k-2})$ for $C\cup E_2\cup\dots\cup E_{\alpha-2}\cup E$ to be a bipartite subdigraph of $C_n(\{1,k\})$.
This is indeed the case since a path in $C\cup E_2\cup\dots\cup E_{\alpha-2}$ from $x_{(\alpha-2)k+1}$ to $x_{\alpha k-2}$ is given by
$x_{(\alpha-2)k+1}\ri x_{(\alpha-2)k+2}\ri \dots \ri x_{(\alpha -1)k-2}\Ri x_{\alpha k-2}$, whose length is $k-2$, an odd number.
This completes this subcase.

\begin{figure}
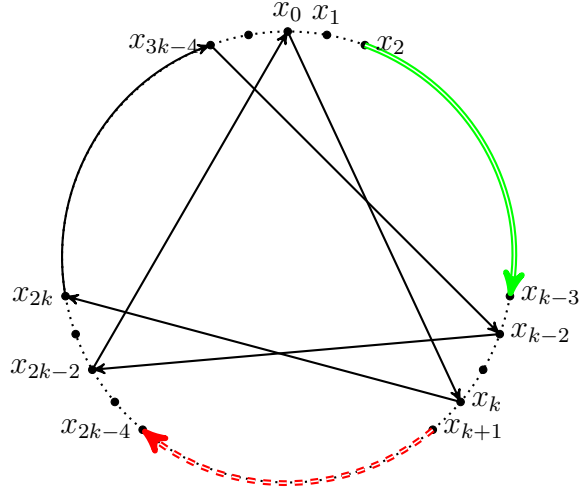

\begin{center}
\caseOneThreeTwo
\end{center}
\caption{(Subcase 1.3.2) A sample digraph, together with the cycle $C$ (black) and the paths $Q_1$ (double, green) and $Q_2$ (double dashed, red).}
\label{fig:case_one_three_two}
\end{figure}

\medskip
\noindent {\it  Subcase 1.3.2. $\alpha = 2.$}

In that case, we get $n=3k-2$ and, since $n\ge 9$, $k\ge 4$.   

Consider the following cycle (see Figure~\ref{fig:case_one_three_two}):
$$C=x_0 \Ri x_k \Ri x_{2k}\ri x_{2k+1}\ri \ldots \ri x_{3k-4}\Ri x_{k-2}\Ri x_{2k-2}\Ri x_0.$$
Since $|A(C)|=k+1$, $C$ is an even cycle. Let $c$ be a 2-colouring  of $C$ with  $c(x_0)=1$. 

We consider the two following paths of $C_n(\{1,k\})$ (see Figure~\ref{fig:case_one_three_two}):
\begin{align*}
Q_1 &=x_2\ri \ldots \ri x_{k-3}\\
Q_2 &=x_{k+1}\ri \ldots \ri x_{2k-4}
\end{align*}

We first add to $C$ the two ears 
$E_1= x_{2k}\Ri Q_1 \ri x_{k-2}$ and
$E_2=x_k\ri Q_2\Ri x_{3k-4}$.
Since the length of both paths $Q_1$ and $Q_2$ is $k-5$, an even number, we need to have
$c(x_{2k})=c(x_{k-2})$ and $c(x_k)=c(x_{3k-4})$ for $C\cup E_1\cup E_2$ to be a bipartite subdigraph of $C_n(\{1,k\})$.
This is indeed the case since $x_0\Ri x_k\Ri x_{2k}$ and $x_{k-2}\Ri x_{2k-2}\Ri x_0$ are paths in $C$.
We can thus extend the 2-colouring $c$ to $C\cup E_1\cup E_2$.

There are now five remaining vertices, namely $x_1$, $x_{k-1}$, $x_{2k-3}$, $x_{2k-1}$ and $x_{3k-3}$, not included in $C\cup E_1\cup E_2$.
We then sequentially add the fives ears
\begin{align*}
E_3 &=x_0\ri x_1\Ri x_{k+1},\\
E_4 &=x_{2k-2}\ri x_{2k-1} \Ri x_1,\\
E_5 &=x_{k-2}\ri x_{k-1} \Ri x_{2k-1},\\
E_6 &=x_{3k-4}\ri x_{3k-3} \Ri x_{k-1}, \text{ and}\\
E_7 &=x_{2k-4}\ri x_{2k-3} \Ri x_{3k-3},
\end{align*}
getting at each step a bipartite subdigraph of $C_n(\{1,k\})$, so that the 2-colouring $c$ can be sequentially extended.
This is indeed the case since 
$x_0\Ri x_k\ri x_{k+1}$, 
$x_{2k-2}\Ri x_0\ri x_1$, 
$x_{k-2}\Ri x_{2k-2}\ri x_{2k-1}$,
$x_{3k-4}\Ri x_{k-2}\ri x_{k-1}$ and
$x_{2k-4}\Ri x_{3k-4}\ri x_{3k-3}$
are paths in $C\cup E_2$, $C\cup E_3$, $C\cup E_4$, $C\cup E_5$ and $E_2 \cup E_6$,
respectively, and thus the endvertices of the five above defined ears have the same colour.

We hence get that $C\cup E_1\cup \dots\cup E_7$ is
a strong bipartite spanning subdigraph  of $C_n(\{1,k\})$, which completes this subcase.

\begin{figure}
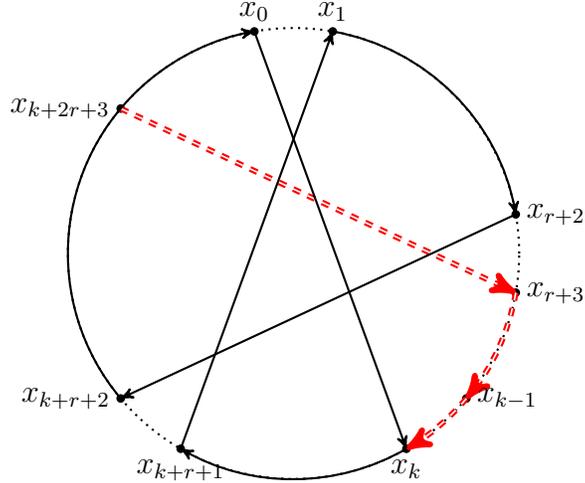

\begin{center}
\caseTwo
\end{center}
\caption{(Case 2) A sample digraph, together with the cycle $C$ (black) and the ear  $E$ (double dashed, red).}
\label{fig:case_two}
\end{figure}

\medskip
\noindent {\bf Case 2. $2k+1\le n\le 3k-4$}

The assumptions of  this case imply $n=2k+r$ with $1\le r\le k-4$. 
Let us consider the following cycle (see Figure~\ref{fig:case_two}):
$$C=x_0\Ri x_k\ri  \ldots \ri x_{k+r+1}\Ri x_1\ri \ldots \ri x_{r+2}\Ri x_{k+r+2}\ri\ldots \ri x_0.$$

Observe that the vertices of $C_n(\{1,k\})$ that are not in $C$ form a path 
$Q=x_{r+3}\ri \ldots \ri x_{k-1}$ (since $r\le k-4$, $Q$ has at least one vertex)
and that $C$ is an even cycle, since $|V(C)|=n-|V(Q)|=n-(k-r-3)=k+2r+3$.

We then add to $C$ the ear $E=x_{k+2r+3}\Ri Q\ri x_k$. 
Since $d_C(x_{k+2r+3},x_k)=k-r-2$ and  $|E|=k-r-2$, $C\cup E$ is a strong bipartite spanning  subdigraph of $C_n(\{1,k\})$, which completes this subcase.

\medskip
\noindent {\bf Case 3. $k+2\le n\le 2k-1$.}

In this case we will construct an even cycle $C$ by concatenating several paths of length $n-k+1$. 
Our assumptions imply $3\le n-k+1\le k$. 
Let $n=(n-k+1)t+s$ with $0\le s\le n-k$. 

Observe that if $t=1$, then $s=k-1$. 
In that case, since $s\le n-k$, we have $s=k-1\le n-k$. On the other hand, we have  $n-k+1\le k$, that is, $k-1\ge n-k$. 
We finally get $s=n-k$ if $t=1$.  

We will consider five subcases.

\begin{figure}
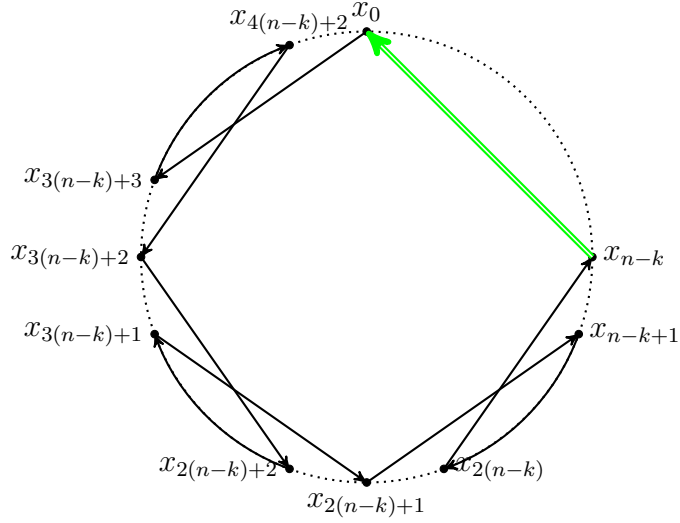

\begin{center}
\caseThreeOne
\end{center}
\caption{(Subcase 3.1) A sample digraph for $t=3$, together with the cycle $P_0\cup P_1\cup P_2$ (black) and the path $P$ (double, green).}
\label{fig:case_three_one}
\end{figure}

\medskip
\noindent {\bf Subcase 3.1.} {\it $t\ge 1$ and $s= n-k.$}

In that case, $n=t(n-k+1)+n-k$. Since $n$ is odd, it follows that $t$ is also odd. Consider the following paths (see Figure~\ref{fig:case_three_one}):
\begin{align*}
P_i&= x_{ik-(i-1)n-i}\Ri x_{(i+1)k-in-i}\ri \ldots \ri x_{ik-(i-1)n-(i+1)}\Ri x_{(i+1)k-in-(i+1)},\\
&\ \ \ \ \text{for every } i\in\{0,\ldots , t-1\},\\
P &= x_{n-k}\Ri  x_0.
\end{align*}

\noindent 

Observe that $C=\bigcup_{i=0}^{t-1}P_i\cup P$ is an even cycle,  since $|P_i|=n-k+1$ for every $i$ and $|P|=1$. 
Furthermore, the vertices of $C_n(\{1,k\})$ not belonging to $C$ form 
a path $Q= x_1\ri \ldots \ri x_{n-k-1}$.

We then add to $C$ the ear $E=x_{n-k+1}\Ri x_1\ri \ldots \ri x_{n-k-1}\ri x_{n-k}$. 
Since $d_C(x_{n-k+1}, x_{n-k})=n-k$ and $|E|=n-k+2$ (these two values have thus the same parity), we get that $C\cup E$ is a strong bipartite spanning subdigraph of $C_n(\{1,k\})$, which completes this subcase.

\begin{figure}
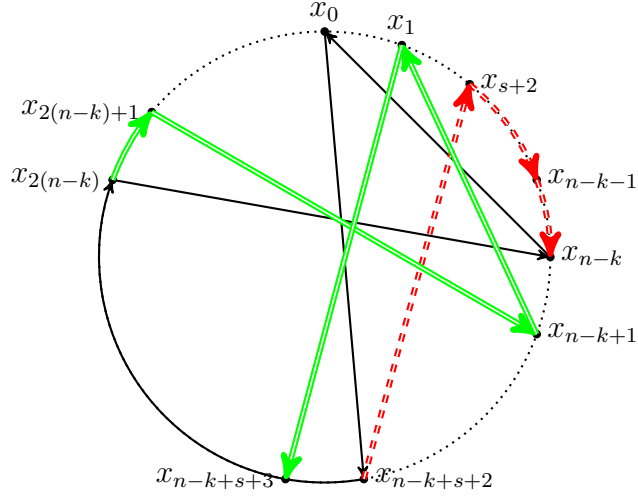

\begin{center}
\caseThreeTwo
\end{center}
\caption{(Subcase 3.2) A sample digraph, together with the cycle $C$ (black) and the ears $E$ (double dashed, red) and $E_0$ (double, green).}
\label{fig:case_three_two}
\end{figure}

\medskip
\noindent {\bf Subcase 3.2.} {\it  $t=2$ and  $0\le s\le n-k-2.$}

In that case, we have $n=2(n-k+1)+s$, and thus $s$ must be odd, so that $1\le s\le n-k-3$. 
Let us consider the following cycle (see Figure~\ref{fig:case_three_two}):
$$C=x_0\Ri x_{n-k+s+2}\ri \ldots \ri x_{2(n-k)}\Ri x_{n-k}\Ri x_0.$$
Note that $|A(C)|=n-k-s+1$, and thus $C$ is an even cycle.

We first add to $C$ the ear 
$$E=x_{n-k+s+2}\Ri x_{s+2}\ri \ldots \ri x_{n-k-1}\ri x_{n-k}.$$
Since $d_C(x_{n-k+s+2}, x_{n-k})=2$ and $|E|=n-k-s-1$ is even, 
$C\cup E$ is a strong bipartite subdigraph of $C_n(\{1,k\})$.

Now, the vertices of $C_n(\{1,k\})$ not belonging to $C\cup E$ can be partitioned into the following three sets,
\begin{align*}
&\{x_{2(n-k)+1}, x_{2(n-k)+2}, \ldots ,x_{2(n-k)+1+s}\}\\
&\{x_1, x_2, \ldots ,x_{s+1}\}, \text{ and }\\
&\{x_{n-k+1},x_{n-k+2},\ldots, x_{n-k+s+1}\},
\end{align*}
each containing $s+1$ vertices.
 
We now add sequentially the following $s+1$ ears:
\begin{align*}
E_i=x_{2(n-k)+i}\ri x_{2(n-k)+1+i}\Ri x_{n-k+1+i}\Ri x_{1+i}\Ri x_{n-k+s+3+i}, 
\text{ for every } i\in\{0,\ldots, s\}.
\end{align*}

Note that each of these ears is of length~4 and contains exactly one vertex of each set.
Since the distance between $x_{n-k+s+3+i}$ and $x_{2(n-k)+i}$ 
in $C\cup E\cup E_0\cup\dots\cup E_{i-1}$, for every $i\in\{0,\ldots, s\}$, 
is $n-k-s-3$, an even number, we get that $C\cup E\cup E_0\cup\dots\cup E_{s}$
is a strong bipartite spanning subdigraph of $C_n(\{1,k\})$, which completes this subcase.

\medskip 
\noindent {\bf Subcase 3.3.} {\it  $t=2$ and $s=n-k-1.$}

In that case, we have $n=3(n-k)+1$.
We claim that $C_n(\{1,k\})$ is isomorphic to $C_n(\{1,3\})$, the digraph considered in {\bf Case 1}. 

Observe first that the chords of $C_n(\{1,k\})$ form the cycle $C_n$. 
Indeed, let $D'$ be the digraph obtained from $C_n(\{1,k\})$ by reversing every arc which is not a chord. Observe that $D'$ is isomorphic to $C_n(\{1,r\})$, with $r=n-k$. Since $n=3r+1$, we get that $n$ and  $r$ are relatively prime. 
Therefore, the chords of $D'$ form a cycle that goes through all vertices of $D'$.
Let  $C'$ be the cycle of $C_n(\{1,k\})$ induced by the chords. The distance on $C'$ between any two consecutive vertices of $C$ is~3, and thus $C_n(\{1,k\})$ 
and $C_n(\{1,3\})$ are isomorphic digraphs, which completes this subcase.

\begin{figure}
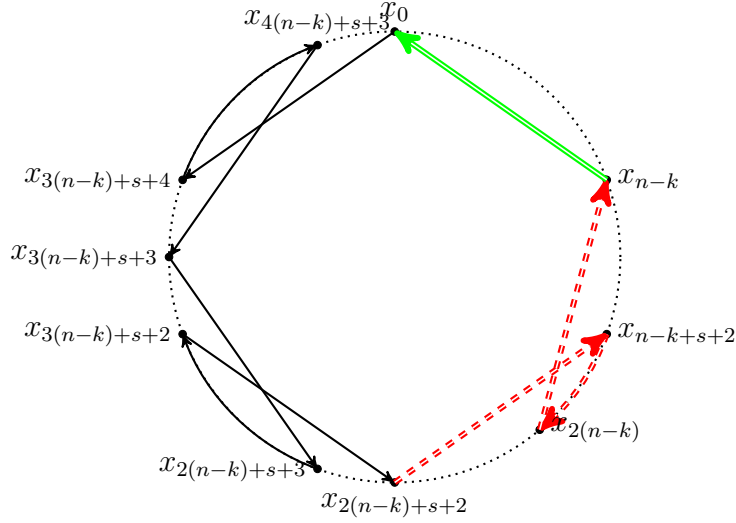

\begin{center}
\caseThreeFour
\end{center}
\caption{(Subcase 3.4) A sample digraph for $t=4$, together with the cycle $C=P_0\cup P_1$ (black) and the paths $P$ (double dashed, red) and $P'$ (double, green).}
\label{fig:case_three_four}
\end{figure}

\medskip
\noindent {\bf Subcase 3.4.} {\it $t\ge 3$ and $0\le s\le n-k-2.$}

Consider the following paths (see Figure~\ref{fig:case_three_four}):
\begin{align*}
P_i &=x_{(t-i)(n-k+1)+s}\Ri x_{(t-i-1)(n-k+1)+s+1}\ri \ldots \ri x_{(t-i)(n-k+1)+s-1}\Ri x_{(t-i-1)(n-k+1)+s},\\
  &\ \ \ \ \text{ for every } i\in\{0,\ldots , t-3\},\\
P &= x_{2n-2k+s+2}\Ri x_{n-k+s+2}\ri \ldots \ri x_{2n-2k}\Ri x_{n-k}, \text{ and}\\
P' &=x_{n-k}\Ri x_0.
\end{align*}

Note that $|P'|=1$, $|P|=n-k-s$ and $|P_i|=n-k+1$. 
Since $n$ is odd and $n=(n-k+1)t+s$, it follows that $t$ and $s$ are of different parity. 
Hence, since $|A(C)|=(t-2)(n-k+1)+n-k-s+1$, we get that  $C=\bigcup_{i=0}^{t-3} P_i\cup P\cup P'$ is an even cycle.  Let $c$ be a 2-colouring  of $C$ with $c(x_0)=1$. 

We are now ready to add ears to the cycle $C$ in order to get a spanning subdigraph of $C_n(\{1,k\})$. 
We consider the following three paths of $C_n(\{1,k\})$:
\begin{align*}
Q_1&=x_{2n-2k+1}\ri \ldots \ri x_{2n-2k+s+1},\\
Q_2&=x_{n-k+1}\ri\ldots \ri x_{n-k+s+1}, \text{ and}\\
Q_3&=x_1\ri \ldots \ri x_{n-k-1}.
\end{align*}

Observe  that $V(C_n(\{1,k\}))=V(C)\cup V(Q_1)\cup V(Q_2)\cup V(Q_3)$. 
We now sequentially add the following three ears to $C$:
\begin{align*}
E_1&=x_{3n-3k+1}\Ri Q_1\ri x_{2n-2k+s+2},\\
E_2&=x_{2n-2k+1}\Ri Q_2\ri x_{n-k+s+2}, \text{ and}\\
E_3&=x_{n-k+1}\Ri Q_3\ri x_{n-k}.
\end{align*}

Since
\begin{align*}
&d_C(x_{3n-3k+1},x_{2n-2k+s+2})=d_{P_{t-3}}(x_{3n-3k+1},x_{2n-2k+s+2})=s+2 \text{ and }|E_1|=s+2,\\
&d_{C\cup E_1}(x_{2n-2k+1},x_{n-k+s+2})=s+2 \text{ and } |E_2|=s+2,\\
&d_C(x_{n-k+1},x_{n-k})=d_{E_2\cup P}(x_{n-k+1},x_{n-k})=n-k \text{ and } |E_3|=n-k,
\end{align*}
the so-obtained subdigraph $C\cup E_1\cup E_2\cup E_3$ is a strong bipartite spanning subdigraph of $C_n(\{1,k\})$.
This completes this subcase.

\begin{figure}
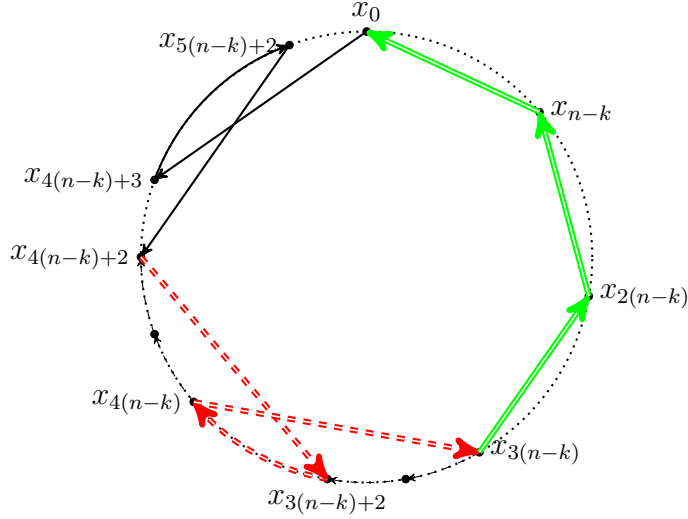

\begin{center}
\caseThreeFive
\end{center}
\caption{(Subcase 3.5) A sample digraph for $t=4$, together with the paths $C=P_0$ (black), $P$ (double dashed, red) and $P'$ (double, green).}
\label{fig:case_three_five}
\end{figure}

\medskip
\noindent {\bf Subcase 3.5.} {\it  $t\ge 3$ and $s= n-k-1.$}

In this case, $t$ must be even since $n$ is odd. We thus have $t\ge 4$. 
Consider the following paths (see Figure~\ref{fig:case_three_five}):
\begin{align*}
P_i &=x_{(t-i+1)(n-k)+t-i-1}\Ri x_{(t-i)(n-k)+t-i-1}\ri \ldots \ri x_{(t-i+1)(n-k)+t-i-2}\Ri x_{(t-i)(n-k)+t-i-2},\\
&\ \ \ \ \text{ for every } i\in\{0,\ldots , t-4\},\\
P &= x_{4n-4k+2}\Ri x_{3n-3k+2}\ri \ldots \ri x_{4n-4k}\Ri x_{3n-3k}, \text{ and}\\
P' &=x_{3n-3k}\Ri x_{2n-2k}\Ri x_{n-k}\Ri x_0.
\end{align*}

Let $C=\bigcup_{i=0}^{t-4}P_i\cup P\cup P'$. 
Since $|P_i|=n-k+1$, $|P|=n-k$ and$|P'|=3$, we get $|A(C)|=(t-3)(n-k+1)+ n-k+3$
and thus  $C$ is an even cycle. 

We now sequentially add the five following ears to $C$: 
\begin{align*}
E_1 &=x_{5n-5k+1}\Ri x_{4n-4k+1}\ri x_{4n-4k+2},\\
E_2 &=x_{4n-4k+1}\Ri x_{3n-3k+1}\ri x_{3n-3k+2},\\
E_3 &=x_{3n-3k+1}\Ri x_{2n-2k+1}\ri \ldots \ri x_{3n-3k-1}\ri x_{3n-3k},\\
E_4 &=x_{2n-2k+1}\Ri x_{n-k+1}\ri \ldots \ri x_{2n-2k-1}\ri x_{2n-2k}, \text{ and}\\
E_5 &=x_{n-k+1}\Ri x_{1}\ri \ldots \ri x_{n-k-1}\ri x_{n-k}.
\end{align*}

Note that the ends of the ear $E_1$ belong to $C$, while the ends of each ear $E_i$,
$2\le i\le 5$, belong to $C\cup E_1\cup \dots \cup E_{i-1}$. 

Moreover, since
\begin{align*}
&d_C(x_{5n-5k+1},x_{4n-4k+2})=d_{P_{t-4}}(x_{5n-5k+1},x_{4n-4k+2})=2 \text{ and } |E_1|=2,\\
&d_{C\cup E_1}(x_{4n-4k+1},x_{3n-3k+2})=2 \text{ and }  |E_2|=2,\\
&d_{C\cup E_1\cup E_2}(x_{3n-3k+1},x_{3n-3k})=n-k \text{ and }  |E_3|=n-k,\\
&d_{C\cup E_1\cup E_2\cup E_3}(x_{2n-2k+1},x_{n-k})=n-k \text{ and }  |E_4|=n-k,\\
&d_{C\cup E_1\cup E_2\cup E_3\cup E_4}(x_{n-k+1},x_{n-k})=n-k  \text{ and } |E_5|=n-k,
\end{align*}
the so-obtained subdigraph $C\cup E_1\cup E_2\cup E_3\cup E_4\cup E_5$ is a strong bipartite spanning subdigraph of $C_n(\{1,k\})$.
This completes this last subcase.

\medskip
This concludes the proof of Lemma~\ref{lem:odd}.
\end{proof}

As observed before, the proof of Theorem~\ref{thm:main} then directly follows from Corollary~\ref{cor:even-integer} and Lemma~\ref{lem:odd}.


\section{Proper-walk connection of digraphs} \label{sec:walk}

Each properly coloured path is also a properly coloured walk, so $\wc(D)\le \pc(D)$ for every digraph $D$. 
Therefore, similarly as the  proper connection number, the  proper-walk connection number of every digraph $D$ is upper bounded by~3. 
On the other hand, the only digraphs with directed proper-walk connection number~1 are the symmetric complete graphs $\overleftrightarrow{K_n}$. 

In this section, we give some sufficient conditions for a digraph $D$ to have proper-walk connection number at most~2.  
If $D$ has an Eulerian closed walk and an even number of arcs then, going along the Eulerian walk and  alternately assigning colours 1 and~2 to the arcs, we get a properly-walk connected arc-colouring of $D$. 
We thus have the following observation.

\begin{obs}\label{obs:A-even}
If $D$ is a strong digraph with $d^+(v)=d^-(v)$ for every vertex $v\in V(D)$  and $|A(D)|$ even, then $\wc(D)\le 2$.
\end{obs}

If $|V(D)|$ is even and $d_D^+(v)=d_D^-(v)$ for every vertex $v\in V(D)$, then $|A(D)|$ is even, so we also have the following observation.

\begin{obs}\label{obs:V-even}
If $D$ is a strong digraph with $d^+(v)=d^-(v)$ for every vertex $v\in V(D)$  and $|V(D)|$ even, then $\wc(D)\le 2$.
\end{obs}

In the next theorems, we consider the case of Hamiltonian digraphs.

\begin{theorem}
If $D$ is a Hamiltonian digraph with $d_D^+(v)=d_D^-(v)$ for every vertex $v\in V(D)$, then $\wc(D)\le 2$.
\end{theorem}

\begin{proof}
Let $C$ be a Hamiltonian cycle of $D$. 
If $|V(D)|$ is even, then $C$ is a strong spanning bipartite subdigraph of $D$ and so $\wc(D)\le 2$ by Corollary~\ref{cor:bipartite}. 

Suppose that $|V(D)|$ is odd and let $D'=D-A(C)$. 
Note that $d_{D'}^+(v)=d_{D'}^-(v)$ for every vertex $v\in V(D')=V(D)$. 
However, $D'$ may be not strong since, for instance, it could be a vertex-disjoint sum of strong digraphs. 
The condition $d_{D'}^+(v)=d_{D'}^-(v)$ for every vertex $v$ implies that there exists a decomposition of $A(D')$ into arc-disjoint cycles, i.e., $D'=C_1\cup \ldots \cup C_k$, with $A(C_i)\cap A(C_j)=\emptyset $ for every $i,j$, $1\le i<j\le k$. 

Since $V(D')$ is odd, at least one cycle is odd, say $C_1$. Since $C$ is a Hamiltonian cycle, the subdigraph $C\cup C_1$ is strong. Let $v$ be any vertex of $C_1$. 
We colour the arcs of $C$ and of $C_1$ alternately with colours 1 and~2, except the arcs incident with $v$. In $C$ we colour the two arcs incident with $v$  with~1, and in $C_1$  we colour the two arcs incident with $v$ with~2. 
It is not difficult to check that the so-obtained colouring is a properly-walk connected arc-colouring of $C\cup C_1$, which gives $\wc(D)\le\wc(C\cup C_1)=2$.
\end{proof}

Recall that the length of a chord $x_px_q$ in a cycle $C=x_0\dots x_{n-1}x_0$
is $|C[x_p,x_q]|$. 

\begin{theorem}\label{thm:short-chords}
Let $D$ be a Hamiltonian digraph with 
$d_D^+(v)\ge 2$ and 
$d_D^-(v)\ge 2$ for every vertex $v\in V(D)$, and $C$ be a Hamiltonian cycle of $D$. If every chord of $C$ has length at most $\left\lceil  |V(D)|/2\right\rceil$, then $\wc(D)= 2$. 
\end{theorem}

\begin{proof}
Since every chord of $C$ has length at most $\left\lceil  |V(D)|/2\right\rceil$, $D\neq\overleftrightarrow{K_n}$ and so $\wc(D)\ge  2$. 
If $|V(D)|$ is even, then $C$ is a strong spanning bipartite subdigraph of $D$ and so $\wc(D)\le 2$ by Corollary~\ref{cor:bipartite}. 
We thus assume that $|V(D)|$ is odd. We will prove that $D$ contains a spanning subdigraph $D'$ with $\wc(D')=2$. To this aim, we will choose some arcs from $A(D)$ and colour them with two colours in such a way they form a properly connected spanning subdigraph of $D$.  

Let $C=x_0x_1\ldots x_{n-1}x_0$ and $\vf $ be a 2-colouring of the arcs of $C$ such that $\vf(x_0x_1)=\vf(x_1x_2)=1$, $\vf(x_{n-1}x_0)=2$, $\vf(x_kx_{k+1})=1$  for $k$ odd, and  $\vf(x_kx_{k+1})=2$  for $k$ even, for every $k\in \{2,\ldots ,n-2\}$. 
Since the arcs $x_0x_1$ and $x_1x_2$ have the same colour, $(C,\vf)$ is not properly connected yet. 

We consider two cases.

\begin{figure}
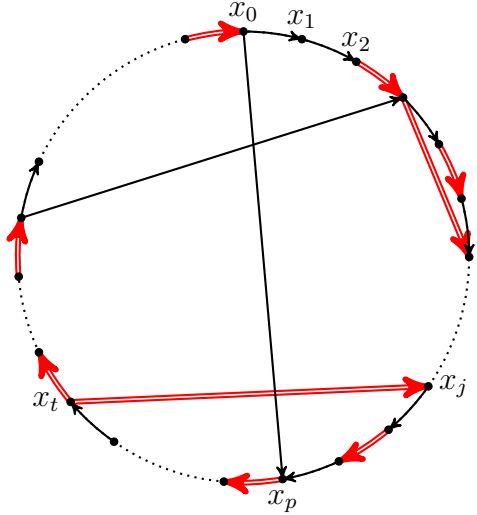

\begin{center}
\sectionThreeOne
\end{center}
\caption{(Theorem~\ref{thm:short-chords}, Case 1) The Hamiltonian digraph $D$ with its arc-colouring; $x_0x_p$ is a shortest even chord of $C$, and $x_tx_j$ an even chord which may not exist.}
\label{fig:section_three_one}
\end{figure}

\medskip
\noindent {\bf Case  1.} {\it $C$ has an even chord.}

Renaming vertices if  necessary, we assume that $x_0x_p$ is an even chord of $C$ with minimum length (see Figure~\ref{fig:section_three_one}). 
We colour $x_0x_p$ with~1, so that $x_0x_px_{p+1}\ldots x_{n-1}$ is a properly connected subdigraph of $D$ (an even cycle).  
Moreover, for every vertex $x_i$ in $\{2,\dots,x_{p-1}\}$, there is a properly connected path $x_ix_0$ whose last arc is $x_{n-1}x_0$.

In order to get a properly-walk connected spanning subdigraph of $D$, it is thus enough to colour some chords of $C$ in such a way that there is a properly coloured $x_0x_i$-walk starting with the arc $x_0x_p$, for every vertex $x_i$ in $\{x_2,\dots,x_{p-1}\}$.

Since $|C[x_0,x_p]|$ is minimum and every chord of $C$ has length at most $\left\lceil  |V(D)|/2\right\rceil$, it follows that every even chord whose head is in $\{x_2,\ldots , x_{p-1}\}$ has its tail in $\{x_{p+2},\ldots ,x_{n-1}\}$. 
Furthermore, since $d_D^-(v)\ge 2$ for every $v\in V(D)$, every vertex in $\{x_2,\ldots, x_{p-1}\}$ is the head of a chord. 

Let $j$ be the smallest integer in $\{x_2,\ldots, x_{p}\}$ such that $x_j$
is the head of an even chord (we may have $j=p$).
Note that $\vf(x_{t-1}x_t)=\vf(x_jx_{j+1})$. 
If $j<p$, we colour $x_tx_j$ with the colour different from $\vf(x_{t-1}x_t)$, so that we now have a properly coloured $x_0x_i$-walk for every vertex $x_i\in \{x_j,\ldots, x_{n-1}\}$. 

All vertices in $\{x_2,\ldots , x_{j-1}\}$ are not the head of an even chord, so they are the head of at least one odd chord.
We claim that, one by one, we can colour one such odd chord for every vertex $x_i$ in $\{x_2,\ldots , x_{j-1}\}$,  in such a way that there will be a properly coloured $x_0x_i$-walk (starting with the arc $x_0x_p$), which will allow us to conclude this case. 

Let $x_sx_2$ be an odd chord. Since $|C[x_s,x_2]|\le (n-1)/2$, we have $p\le s\le n-1$ and, since $x_sx_2$ is odd, $\vf(x_{s-1}x_s)=1$. 
We then colour $x_sx_2$ with~2, so that the walk $x_0x_p\ldots x_sx_2$ is properly coloured. 

Assume now that for every $i\in \{2,\ldots, \ell-1\}$, $\ell < j-1$, we have found a properly coloured $x_0x_i$-walk starting with $x_0x_p$ and ending with an odd chord. 
Let $x_qx_{\ell}$ be an odd chord. Since $|C[x_q,x_{\ell}]|\le (n-1)/2$, we have either $p+1\le q\le n-1$ or $0\le q \le \ell-3$. 
If $p+1\le q\le n-1$, then we colour $x_qx_{\ell}$ with the colour different from $\vf(x_{q-1}x_q)$, so that $x_0x_p\ldots x_qx_{\ell}$ is a properly coloured walk. 
Otherwise,
if $q=0$, then we colour $x_qx_{\ell}$ with~1, 
if $q=1$, then we colour $x_qx_{\ell}$ with~2, 
and if $2\le q \le \ell-3$, then we colour $x_qx_{\ell}$ with the colour different from the colour of the odd chord with head in $x_q$. 
Let $P$ be a properly coloured $x_0x_q$-walk starting with $x_0x_p$ and ending with a coloured chord. Then $Px_qx_{\ell}$ is the walk we are looking for.
This completes this case.

\begin{figure}
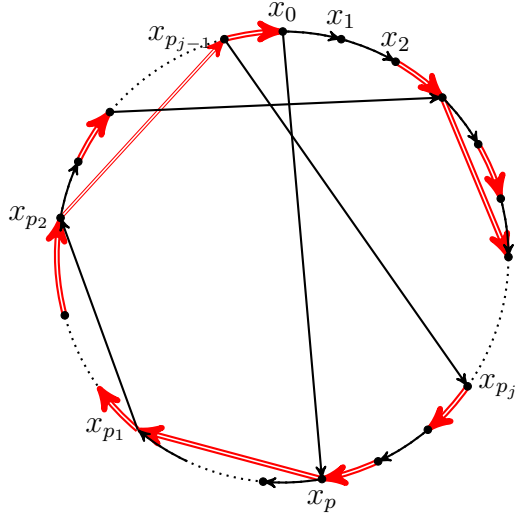

\begin{center}
\sectionThreeTwo
\end{center}
\caption{(Theorem~\ref{thm:short-chords}, Case 2) The Hamiltonian digraph $D$ with its arc-colouring; $x_0x_p$ is a longest odd chord of $C$.}
\label{fig:section_three_two}
\end{figure}

\medskip
\noindent {\bf Case  2.} {\it Every chord in $C$ is odd.}

Renaming vertices if  necessary, we assume that $x_0x_p$ is an odd chord of $C$ with  maximum length (see Figure~\ref{fig:section_three_two}). We colour $x_0x_p$ with 1. 
Next, we choose a path containing only odd chords, $x_p\Ri x_{p_1}\Ri x_{p_2}\Ri \ldots \Ri x_{p_{j-1}}\Ri x_{p_j}$,  until we ``jump over  $x_0$'' (i.e., $p_{j-1}> p_j$ and $p_k < p_{k+1}$ for every $k<p$). 
Since $x_0x_p$ has  maximum length, we have $p_j<p$ (since $n$ is odd, we cannot have $x_{j-1}=x_0$). 
We colour the chords of this path alternately, starting with colour~1 on $x_0x_p$.
Observe that since all the chords are odd,  the colour of $x_{p_i}\Ri x_{p_{i+1}}$ is the same as the colour of $x_{p_{i+1}}\ri x_{p_{i+1}+1}$ for every $i\in \{1,\ldots j-2\}$. 
However, the colour of $x_{p_{j-1}}\Ri x_{p_j}$ is different from the colour of $x_{p_j}\ri x_{p_j+1}$. Thus, the closed walk 
$$x_0\Ri x_p\Ri x_{p_1}\Ri x_{p_2} \Ri \ldots \Ri x_{p_j} \ri x_{p_j+1}\ri \ldots \ri x_{n-1}\ri x_0$$
is properly coloured.

Hence, for every vertex $x_i\in \{x_2,\ldots ,x_{p-1}\}$, there is a  properly coloured $x_ix_0$-walk ending with $x_{n-1}x_0$.
In order to get a properly-walk connected spanning subdigraph of $D$, it is thus enough to colour some chords of $C$ in such a way that there is a properly coloured $x_0x_i$-walk starting with the arc $x_0x_p$, for every vertex $x_i$ in $\{x_2,\dots,x_{j-1}\}$.

Similarly as in {\bf Case 1}, 
we claim that, one by one, we can colour, for every vertex $x_i$ in $\{x_2,\ldots , x_{p_j-1}\}$ one chord with head at $x_i$,  in such a way that there will be a properly coloured $x_0x_i$-walk (starting with $x_0x_p$ and ending with the chosen chord), which will allow us to conclude this case.

First observe that, for every vertex $x_i\in \{x_{p+1},\ldots ,x_{n-1}\}$, there is a properly coloured $x_0x_i$-walk starting with $x_0x_p$ and ending with an arc in $A(C)$. 
Let $x_sx_2$ be an odd chord. 
Since $|C[x_s,x_2]|\le (n-1)/2$, we have $p+1\le s\le n-1$. Since $x_sx_2$ is odd, $\vf(x_{s-1}x_s)=1$. 
We then colour $x_sx_2$ with~2, so that we obtain the properly coloured walk $P_2x_sx_2$, where $P_2$ is the properly coloured $x_0x_s$-walk starting with $x_0x_p$ and ending with an arc in $A(C)$. 

Assume now that for every $i\in \{2,\ldots, \ell-1\}$, $\ell < j-1$, we have found a properly coloured $x_0x_i$-walk starting with $x_0x_p$ and ending with an odd chord. 
Let $x_qx_{\ell}$ be an odd chord. Since $|C[x_q,x_{\ell}]|\le (n-1)/2$, we have either $p+1\le q\le n-1$ or $0\le q \le \ell-3$. 
If $p+1\le q\le n-1$, then we colour $x_qx_{\ell}$ with the colour different from the colour of $x_{q-1}x_q$. Thus $P_{\ell}x_qx_{\ell}$ (where $P_{\ell}$  is the properly coloured $x_0x_q$-walk starting with $x_0x_p$   and  ending with an arc in $A(C)$) is a properly coloured walk. 
Otherwise,
if $q=0$, then we colour $x_qx_{\ell}$ with 1,
if $q=1$, then we colour $x_qx_{\ell}$ with 2, and
ff $2\le q \le \ell-3$, then we colour $x_qx_{\ell}$ with the colour different from the colour of the odd chord with head $x_q$. 
Let $P$ be a properly coloured $x_0x_q$-walk starting with $x_0x_p$ and ending with a coloured chord. Then $Px_qx_{\ell}$ is a path that we  are looking for.

This completes the proof of Theorem~\ref{thm:short-chords}.
\end{proof}

\section{Concluding remarks}\label{sec:remarks}

In Section~\ref{sec:walk} we have proved that   if $C$ is a Hamiltonian cycle of a digraph $D$ satisfying $d_D^+(v)\ge 2$ and $d_D^-(v)\ge 2$ for every vertex $v$ of $D$, then $\wc(D)\le 2$, provided that every chord of $C$ has length at most $\left\lceil  |V(D)|/2\right\rceil$. We expect that the condition on the length of chords can be removed. We thus propose the following conjecture.

\begin{conjecture}
If $D$ is a  Hamiltonian digraph  such that 
$d_D^+(v)\ge 2$ and 
$d_D^-(v)\ge 2$ for every vertex $v\in V(D)$, then $\wc(D)\le 2$.
\end{conjecture}

In Section~\ref{sec:path}, we gave some necessary conditions for a Hamiltonian digraph to have  proper connection number at most~2. 
In fact, we proved  stronger results. Namely, we proved that such a  digraph has a spanning strong bipartite subdigraph. 
Bang-Jensen {\it et al.} studied in \cite{BaBeHaYe18} the existence, in a given digraph $D$, of a spanning bipartite subdigraph with some special properties.
A {\it $2$-partition} of a digraph $D$ is a partition of its set of vertices in two parts.
If $(V_1,V_2)$ is a 2-partition of a digraph $D$, then the bipartite digraph induced by the set of arcs having one end in each part is denoted by $B_D(V_1,V_2)$. 
A 2-partition $(V_1,V_2)$ of $D$ is {\it strong} if $B_D(V_1,V_2)$ is strong. 
The authors of~\cite{BaBeHaYe18}, Bang-Jensen {\it et al.}  studied the complexity of determining whether 2-partitions with some restrictions on the degree of vertices,
or strong 2-partitions, exist in digraphs. 
Among others results, they proved the following theorem.

\begin{theorem}\label{thm:two-partition}
It is NP-complete to decide whether a strong digraph has a strong $2$-partition.
\end{theorem}

From the proofs of Theorem~\ref{thm:even_tail} and Corollary~\ref{cor:even_head}, and the observation that a Hamiltonian digraph of even order has a strong 2-partition, the following result follows.

\begin{theorem}\label{thm:chords_tail}
Let $D$ be a Hamiltonian digraph  and $C$ be a Hamiltonian cycle of $D$. 
If for every $v\in V(D)$ there is an even chord of $C$ with tail $v$, 
then $D$ has a strong $2$-partition.
\end{theorem}

\begin{theorem}\label{thm:chords_head}
Let $D$ be a Hamiltonian digraph  and $C$ be a Hamiltonian cycle of $D$. 
If for every $v\in V(D)$  there is an even chord of $C$ with head $v$, 
then $D$ has a strong $2$-partition.
\end{theorem}

In the proof of Lemma~\ref{lem:odd}, we have shown that, for $n\ge 9$ and $k$ odd,  every circulant digraph $C_n(\{1,k\})$ has a strong 2-partition. One can see that $C_5(\{1,3\})$ has a strong 2-partition. 
Furthermore, using Theorems \ref{thm:chords_tail} and~\ref{thm:chords_head}, we obtain the following result.

\begin{theorem}
If $n\neq 7$, $S\subseteq\{1,\dots,n-1\}$, $|S|\ge 2$, and $1\in S$,  then $C_n(S)$ has a strong $2$-partition.
\end{theorem}

It would be interesting to consider the case of Hamiltonian digraphs with respect to Theorem~\ref{thm:two-partition}. We can thus ask the following question.

\begin{question}
What is the complexity of deciding whether a Hamiltonian digraph $D$ has a strong $2$-partition?
\end{question}

 
\end{document}